\documentclass[reqno]{amsart}
\usepackage{amssymb,amsmath,amsfonts,bm}
\usepackage{dsfont}
\usepackage{epsfig}
\usepackage{graphicx}
\usepackage{esint}
\usepackage{enumerate}
\usepackage{verbatim}
\usepackage{color}
\usepackage{hyperref}
\usepackage{caption}
\usepackage{geometry}


\date{} 

\theoremstyle{plain}
\newtheorem{theorem}{Theorem}[section]
\newtheorem{definition}[theorem]{Definition}
\newtheorem{proposition}[theorem]{Proposition}

\newtheorem{lemma}[theorem]{Lemma}
\newtheorem{corollary}[theorem]{Corollary}
\newtheorem{remark}[theorem]{Remark}

\numberwithin{theorem}{section}
\numberwithin{equation}{section}
\numberwithin{figure}{section}

\let\oldtocsection=\tocsection
 
\let\oldtocsubsection=\tocsubsection
 
\let\oldtocsubsubsection=\tocsubsubsection
 
\renewcommand{\tocsection}[2]{\hspace{0em}\oldtocsection{#1}{#2}}
\renewcommand{\tocsubsection}[2]{\hspace{1em}\oldtocsubsection{#1}{#2}}
\renewcommand{\tocsubsubsection}[2]{\hspace{2em}\oldtocsubsubsection{#1}{#2}}

\DeclareMathSizes{10}{9}{6}{5}
\begin{document}

\parskip=4pt

\vspace*{1cm}
\title[Global well-posedness of a  binary-ternary Boltzmann equation]
{Global well-posedness of a binary-ternary Boltzmann equation}
\author[Ioakeim Ampatzoglou]{Ioakeim Ampatzoglou}
\address{Ioakeim Ampatzoglou,  
Courant Institute of Mathematical Sciences, New York University.}
\email{ioakampa@cims.nyu.edu}

\author[Irene M. Gamba]{Irene M. Gamba}
\address{Irene M. Gamba,  
Department of Mathematics, The University of Texas at Austin.}
\email{gamba@math.utexas.edu}
\author[Nata\v{s}a Pavlovi\'{c}]{Nata\v{s}a Pavlovi\'{c}}
\address{Nata\v{s}a Pavlovi\'{c},  
Department of Mathematics, The University of Texas at Austin.}
\email{natasa@math.utexas.edu}

\author[Maja Taskovi\'{c}]{Maja Taskovi\'{c}}
\address{Maja Taskovi\'{c},  
Department of Mathematics, Emory University.}
\email{maja.taskovic@emory.edu}
\vspace*{-3cm}
\begin{abstract}
In this paper we show global well-posedness near vacuum for the binary-ternary Boltzmann equation.   The binary-ternary Boltzmann equation provides a correction term to the classical Boltzmann equation, taking into account both binary and ternary interactions of particles, and may serve as a more accurate description model for denser gases in non-equilibrium. Well-posedness of the classical Boltzmann equation and, independently, the purely ternary Boltzmann equation follow as special cases. To prove global well-posedness, we use a Kaniel-Shinbrot iteration and related work to approximate the solution of the nonlinear equation by monotone sequences of supersolutions and subsolutions. This analysis required establishing  new convolution type estimates to control the contribution of the ternary collisional operator to the model. We show that the ternary operator allows consideration of softer potentials than  the one  binary operator, consequently our solution to the ternary correction of the Boltzmann equation preserves all the properties of the binary interactions solution. These results are novel for collisional operators of monoatomic gases with either hard or soft potentials that model both binary and ternary interactions. 
\end{abstract}
\maketitle

\section{Introduction}\label{sec:intro}
We study global in time well-posedness near vacuum of the Cauchy problem for an extension of the classical Boltzmann transport equation (BTE) for monoatomic binary interactions gases that includes ternary interactions. This equation, which can be viewed as a model of a denser gas dynamics,  has been recently  introduced by two of the authors in  \cite{ours}, who rigorously derived, from finitely many particle dynamics,  the  purely ternary model for the case of hard potential interactions zone for short times. Moreover, it is seen in \cite{thesis} that the ternary collisional operator derived in \cite{ours} has the same conservation laws and entropy production properties as the classical binary operator, which justifies that the introduced ternary term  can serve as a higher order correction  to the Boltzmann equation.  Such rigorous derivation of the full binary-ternary model  is a work in progress \cite{future AP}. Let us also mention that Maxwell models with multiple particle interactions have been studied in \cite{multiple gamba 2, multiple gamba},  for the space homogeneous case via Fourier Transform methods.

In this paper, we provide the first rigorous analytical result that shows  global in time existence and uniqueness of  {\em mild solutions} near vacuum to the binary-ternary model and the purely ternary model on its own. By mild solutions we mean  that the  $x$-space dependence of the solution is evaluated along the characteristic curves given by the Hamiltonian evolution of the particle system in between collisions  (denoted by $f^{\#}$, being introduced in Subsection \ref{subsec: notation and solution}). The analytical techniques we use are inspired by the works  \cite{kaniel-shinbrot,illner-shinbrot,beto85,to86,to88,pato89} and the more recent work of \cite{alonso,alonso-gamba}. These techniques rely on finding convergent supersolutions and subsolutions to the strong form of \eqref{generalized boltzmann equation} in the associated strong topology of space-velocity Maxwellian weighted in $L^\infty$-functions.

The binary-ternary Boltzmann transport equation we focus on is given by
\begin{equation}
\begin{cases}\label{generalized boltzmann equation}
\partial_tf+v\cdot\nabla_xf=Q_2(f,f)+Q_3(f,f,f),\quad (t,x,v)\in (0,\infty)\times\mathbb{R}^d\times\mathbb{R}^d,\\
f(0)=f_0,\quad (x,v)\in\mathbb{R}^d\times\mathbb{R}^d,
\end{cases}
\end{equation}
and describes the evolution of the probability density $f$ of a dilute gas
 in non-equilibrium in $\mathbb{R}^d$, $d\geq 2$, given an initial condition $f_0:\mathbb{R}^d\times\mathbb{R}^d\to\mathbb{R}$, when both binary and ternary interactions among particles can occur.
 The operator  $Q_2(f,f)$  is the classical binary collisional operator, which expresses  binary  elastic interactions between particles, and is of quadratic order, while the operator $Q_3(f,f,f)$ is the ternary collisional operator which expresses ternary interactions among particles, and is of cubic order. For the exact forms of the operators $Q_2(f,f)$, $Q_3(f,f,f)$ used in this paper, see \eqref{binary kernel}, \eqref{ternary kernel} respectively. We should mention that the purely ternary model, rigorously derived for short times in \cite{ours}, is given by 
\begin{equation}
\begin{cases}\label{intro-ternary}
\partial_tf+v\cdot\nabla_xf=Q_3(f,f,f),\quad (t,x,v)\in (0,\infty)\times\mathbb{R}^d\times\mathbb{R}^d,\\
f(0)=f_0,\quad (x,v)\in\mathbb{R}^d\times\mathbb{R}^d,
\end{cases}
\end{equation}
We refer to \eqref{intro-ternary} as the ternary Boltzmann transport equation.

For the classical Boltzmann transport equation
\begin{equation}\label{intro-binary}
\begin{cases}
\partial_tf+v\cdot\nabla_xf=Q_2(f,f),\quad (t,x,v)\in (0,\infty)\times\mathbb{R}^d\times\mathbb{R}^d,\\
f(0)=f_0,\quad (x,v)\in\mathbb{R}^d\times\mathbb{R}^d,
\end{cases}
\end{equation}
 one way to obtain global well-posedness near vacuum is by utilizing an iterative scheme which constructs monotone sequences of supersolutions and subsolutions  that converge to the global solution of \eqref{intro-binary}. This has been carried out for the first time by Illner and Shinbrot  \cite{illner-shinbrot}, who were motivated by the work of Kaniel and Shinbrot \cite{kaniel-shinbrot}, who in turn showed local in time well-posedness for \eqref{intro-binary} following this program.  Later, this work was extended to include wider range of potentials and to relax assumptions on initial data by Bellomo and Toscani \cite{beto85} , Toscani \cite{to86, to88} and Palczewski and Toscani \cite{pato89}.   Alonso and Gamba \cite{alonso-gamba}  used Kaniel-Shinbrot iteration  to derive distributional and classical solutions to \eqref{intro-binary} for soft potentials for large initial data near two sufficiently close Maxwellians in position and velocity space, while Alonso \cite{alonso} used this technique to study the inelastic Boltzmann equation for hard spheres. Strain \cite{st10} remarks that the estimates he derives can be combined with the Kaniel-Shinbrot iteration to obtain existence of  unique mild solution for the relativistic Boltzmann equation.

Kaniel-Shinbrot iteration is also an important tool for proving non-negativity of solutions, see for example \cite{po88, gl06, chyu08}. Also, when initial data has decay in the direction of $x-v$ as opposed to $x$ and $v$ separately, Kaniel-Shinbrot iteration can be used to construct solutions with infinite energy, see for example \cite{mipe97, zhhu08, wezh14}.

Certain problems have been solved by considering modifications of the Kaniel-Shinbrot iteration. For example, Bellomo and Toscani \cite{beto87} adapted the iteration to the Boltzmann-Enskog equation. Ha, Noh and Yun \cite{hanoyu07}  and Ha and Noh \cite {hano10} also modified the iteration to prove global existence of mild solutions to the Boltzmann system for gas mixtures  in the elastic and the inelastic cases respectively.  Also, Wei and Zhang \cite{wezh06} used another modified iteration to obtain eternal solutions for the Boltzmann equation.

 The goal of this paper is to establish  global existence and uniqueness of a mild solution near vacuum  to  the binary-ternary Boltzmann equation \eqref{generalized boltzmann equation} in spaces of non-negative functions bounded by a Maxwellian.   Moreover, solution of \eqref{intro-ternary} follows as a special case. Inspired by \cite{illner-shinbrot, kaniel-shinbrot, alonso-gamba}, we devise an iterative scheme which constructs monotone sequences of supersolutions and subsolutions to \eqref{generalized boltzmann equation}. For small enough initial data, the beginning condition of the iteration holds globally in time and the two sequences can be shown to converge to the same limit, namely a function $f$ which solves equation \eqref{generalized boltzmann equation} in a mild sense. This strategy requires new ideas given the fact that ternary interactions are also taken into account in \eqref{generalized boltzmann equation}. 
 
 In particular, due to the presence of the ternary correction term,  one  needs to properly adapt the iteration, so that the corresponding supersolutions and subsolutions remain monotone and convergent. One of the main tools is stated in Lemma \ref{convolution lemma} which provides important exponentially weighted convolution estimates. This Lemma not only recovers the estimates developed in  \cite{alonso-gamba} for the binary interaction case, but also develops  a new approach  in order  to treat the ternary interaction case.  Lemma \ref{convolution lemma} is crucially used to obtain uniform in time, space-velocity $L^1$-bounds that control  the ternary gain and loss terms ($L^\infty L^1$ estimates). In fact, using Lemma \ref{convolution lemma} one first obtains asymmetric  estimates (see Lemma \ref{bound on R tilde lemma}) because of the asymmetry introduced by the ternary collisional operator which is not present in the binary case. However, to obtain convergence, it is essential to have symmetry with respect to the inputs of the gain and loss operators. We were able to achieve this symmetrization  in Proposition \ref{permutation lemma}. Finally, we also use Lemma \ref{convolution lemma} to prove  a global estimate for the time average of the gain and loss operators along the characteristics of the Hamiltonian, see Proposition \ref{bounds on operators proposition}. With this, we were able to extend the argument for  controlling  the binary time integrals of both, gain  and loss terms, (see \cite{alonso-gamba}),   to the ternary case by invoking   properties of ternary interactions and a $2d$-analog of  the time  integration  bound for a traveling Maxwellian.

With these tools in hand, for small initial data, the constructed iteration scheme is proved to converge to the unique, global in time mild solution  of \eqref{generalized boltzmann equation}. For more details see Section \ref{sec iteration} and Section \ref{sec: gwp}.
\subsection*{Organization of the paper} In Section \ref{sec: towards the statement}, we review the binary and ternary collisional operators and decompose them into gain and loss forms. We then introduce some necessary notation and state our main result (Theorem \ref{gwp theorem}). In Section \ref{sec: properties of gain and loss}, we prove the convolution estimate and derive essential  bounds for the gain and loss operators. In Section \ref{sec iteration}, we inductively construct monotone sequences of  supersolutions and subsolutions  which are shown to converge to a common limit which solves the binary-ternary Boltzmann equation \eqref{generalized boltzmann equation}, as long as a beginning condition is satisfied. Finally, in Section \ref{sec: gwp} we provide the proof of our main result (Theorem \ref{gwp theorem}).
\subsection*{Acknowledgements} I.A.  acknowledges support  from NSF grants DMS-1516228, DMS-1840314, DMS-2009549 and the Simons Collaboration on Wave Turbulence.  I.M.G.  acknowledges support from NSF grants DMS-1107465 and DMS-1715515.   N.P.  acknowledges support from NSF grants DMS-1516228, DMS-1840314 and DMS-2009549. M.T.   acknowledges support from NSF grant DMS-1107465 and an AMS-Simons Travel Grant.
\section{Towards the statement of the main result}\label{sec: towards the statement}
The goal of this section is to present the precise statement of the main result of this paper. In order to do so, we first review the collisional operators and decompose them to gain and loss form  in Subsection \ref{subsec: collisional}, introduce necessary notation and the notion of a solution in Subsection \ref{subsec: notation and solution}, and then state the main result in Subsection \ref{subsec: statement} (Theorem \ref{gwp theorem}).

\subsection{Collisional operators}\label{subsec: collisional}

\subsubsection{Binary collisional operator}
 The  binary collisional operator is given by
\begin{equation}\label{binary kernel}
Q_2(f,f)(t,x,v)=\int_{\mathbb{S}^{d-1}\times\mathbb{R}^d}B_2(u,\omega)\left(f(t,x,v')f(t,x,v_1')-f(t,x,v)f(t,x,v_1)\right)\,d\omega\,dv_1,
\end{equation} 
where
\begin{equation}\label{binary relative velocity}
u:=v_1-v,
\end{equation}
is the relative velocity of a pair of interacting particles centered at $x, x_1\in\mathbb{R}^d$, with velocities $v,v_1\in\mathbb{R}^d$ before the binary interaction    with respect to the impact direction 
\begin{equation}\label{impact direction binary}
\omega:=\frac{x_1-x}{|x-x_1|}\in\mathbb{S}^{d-1},
\end{equation}
and 
\begin{equation}\label{binary collision formulas}
\begin{aligned}
v':=v+(\omega\cdot u)\omega,\quad v_1':=v_1-(\omega\cdot u)\omega,
\end{aligned}
\end{equation}
are the outgoing velocities after the binary interaction.

 One can easily verify  the binary energy-momentum conservation system is satisfied
\begin{align}
v'+v_1'&=v+v_1,\label{binary cons of momentum}\\
|v'|^2+|v_1'|^2&=|v|^2+|v_1|^2.\label{binary conservation of energy}
\end{align}
Either \eqref{binary collision formulas} or  \eqref{binary cons of momentum}-\eqref{binary conservation of energy} imply
\begin{equation}\label{binary rel vel equation}
|u'|=|u|,\quad\text{where }u':=v_1'-v'.
\end{equation}
In addition,  equation \eqref{binary collision formulas} yields the specular reflection with respect to the impact direction $\omega$
\begin{equation}\label{skewsymmetry for binary}
\omega\cdot u'=-\omega\cdot u.
\end{equation}
In fact it is not hard to show that, given $v,v_1\in\mathbb{R}^d$, expression \eqref{binary collision formulas} gives the general solution of the system \eqref{binary cons of momentum}-\eqref{binary conservation of energy}, parametrized by $\omega\in\mathbb{S}^{d-1}$. The factor $B_2$ in the integrand of \eqref{binary kernel} is  referred as the binary interaction differential cross-section which depends on relative velocity $u$ and the impact direction $\omega$. It expresses the transition probability of binary interactions, and we assume it is of the form
\begin{equation}\label{form of B_2}
B_2(u,\omega)=|u|^{\gamma_2}b_2(\hat{u}\cdot\omega),\quad\gamma_2\in(-d+1,1],
\end{equation}
where $\hat{u}=\displaystyle\frac{u}{|u|}\in\mathbb{S}^{d-1}$ is the relative velocity direction and $b_2:[-1,1]\to [0,\infty)$ is the binary angular transition probability density.   It is worth mentioning that the case $\gamma_2\in (0,1]$ corresponds to hard potentials, the case $\gamma_2\in(-d+1,0)$ corresponds to soft potentials and the case $\gamma_2=0$ corresponds to Maxwell molecules.

We assume  that the binary angular transition probability density $b_2$ satisfies the following properties
\begin{itemize}
\item  $b_2:[-1,1]\to\mathbb{R}$ is a measurable, non-negative probability density.
\item $b_2$ is even i.e.
\begin{equation}\label{binary micro2}
b_2(-z)=b_2(z),\quad\forall z\in[-1,1], 
\end{equation}
which, due to  property from \eqref{skewsymmetry for binary}, yields the binary micro-reversibility condition  
\begin{equation}\label{binary micro}
b_2(\hat{u}'\cdot\omega)=b_2(\hat{u}\cdot\omega),\quad\forall\omega\in\mathbb{S}^{d-1},\quad\forall v,v_1\in\mathbb{R}^d,
\end{equation}
where $\hat{u}'=\displaystyle\frac{u'}{|u|}\in\mathbb{S}^{d-1}$  is the scattering direction.
In addition, relations  \eqref{binary rel vel equation}, \eqref{form of B_2} and \eqref{binary micro} yield
\begin{equation}\label{pre-post bin cross}
B_2(u',\omega)=B_2(u,\omega),\quad\forall\omega\in\mathbb{S}^{d-1},\quad\forall v,v_1\in\mathbb{R}^d.
\end{equation}

\item   The probability density is integrable on the sphere, i.e. for any fixed $\hat{u}$ we have $b_2(\hat{u}\cdot\omega) \in {L^1(\mathbb{S}^{d-1})}$, or equivalently 
 $b_2(z) (1-z^2)^{\frac{d-3}{2}} \in L^1([-1,1])$, for $z=\hat{u}\cdot\omega$,  and 
 \begin{equation}\label{Grad binary}
\|b_2\|_ {L^1(\mathbb{S}^{d-1})} \ = |\mathbb{S}^{d-2}|  \int_{-1}^1|b_2(z)|(1-z^2)^{\frac{d-3}{2}}\,dz<\infty,
\end{equation}
where $|\mathbb{S}^{d-2}| $ is the volume of the $(d-2)$-dimensional sphere.
\end{itemize}

 \begin{remark} The integrability  condition  on $b_2$   is weaker than the classical Grad cut-off assumption which assumes $b_2$ is a bounded function of $z=\hat{u}\cdot\omega$.  So our result is valid for a broader class of angular transition probability measures. 
\end{remark} 
\begin{remark} One can see that the usual hard sphere model   is a special case of the form \eqref{form of B_2} for 
$$\gamma_2=1,\quad b_2(z)=\frac{|z|}{2}.   $$
\end{remark}

\subsubsection{Ternary collisional operator} The ternary collisional operator is given by (see \cite{ours} for details)
\begin{align}
&Q_3(f,f,f)(t,x,v)\nonumber\\
&=\int_{\mathbb{S}^{2d-1}\times\mathbb{R}^{2d}}\hspace{-0.1cm}B_3( \bm{u},\bm{\omega})\left(f(t,x,v,v^*)f(t,x,v,v_1^*)f(t,x,v,v_2^*)-f(t,x,v)f(t,x,v_1)f(t,x,v_2)\right)\,d\omega_1\,d\omega_2\,dv_1\,dv_2,\label{ternary kernel}
\end{align} 
where
\begin{equation}\label{ternary relative velocity}
\bm{u}:=\begin{pmatrix}
v_1-v\\v_2-v
\end{pmatrix}\in\mathbb{R}^{2d},
\end{equation}
is the relative velocity of some colliding particles centered at $x,x_1,x_2\in\mathbb{R}^d$, with velocities $v,v_1,v_2\in\mathbb{R}^d$ before the ternary interaction with respect to the impact directions vector
\begin{equation}\label{impact vector}
\bm{\omega}=\begin{pmatrix}
\omega_1\\\omega_2
\end{pmatrix}:=\frac{1}{\sqrt{|x-x_1|^2+|x-x_2|^2}}\begin{pmatrix}
x_1-x\\
x_2-x
\end{pmatrix}\in\mathbb{S}^{2d-1},
\end{equation}
and
\begin{equation}\label{ternary collision formulas}
\begin{aligned}
v^*&=v+\frac{\omega_1\cdot(v_1-v)+\omega_2\cdot (v_2-v)}{1+\omega_1\cdot \omega_2}(\omega_1+\omega_2),\\
v_1^*&=v_1-\frac{\omega_1\cdot(v_1-v)+\omega_2\cdot (v_2-v)}{1+\omega_1\cdot\omega_2}\omega_1,\\
v_2^*&=v_2-\frac{\omega_1\cdot(v_1-v)+\omega_2\cdot(v_2-v)}{1+\omega_1\cdot\omega_2}\omega_2,
\end{aligned}
\end{equation}
are the outgoing velocities of the particles after the ternary interaction.
It can be easily seen that  if $v^*,v_1^*,v_2^*$ are given by \eqref{ternary collision formulas},  the ternary energy-momentum conservation system 
\begin{align}
v^*+v_1^*+v_2^*&=v+v_1+v_2,\label{ternary cons of momentum}\\
|v^*|^2+|v_1^*|^2+|v_2^*|^2&=|v|^2+|v_1|^2+|v_2|^2,\label{ternary conservation of energy}
\end{align}
is satisfied.
Expressions  \eqref{ternary cons of momentum}-\eqref{ternary conservation of energy} also imply the ternary velocities conservation law
\begin{equation}\label{cons of vel magitude ternary}
|v^*-v_1^*|^2+|v^*-v_2^*|^2+|v_1^*-v_2^*|^2=|v-v_1|^2+|v-v_2|^2+|v_1-v_2|^2,
\end{equation}
For the postcollisional relative velocity, we will write
\begin{equation}\label{ternary post relative velocity}
\bm{u^*}:=\begin{pmatrix}
v_1^*-v^*\\v_2^*-v^*
\end{pmatrix},
\end{equation}
and let us also define the quantities
\begin{align}
|\bm{\widetilde{u}}|&:=\sqrt{|v-v_1|^2+|v-v_2|^2+|v_1-v_2|^2}\label{ternary magnitude pre},\\
|\bm{\widetilde{u}^*}|&:=\sqrt{|v^*-v_1^*|^2+|v^*-v_2^*|^2+|v_1^*-v_2^*|^2}\label{ternary magnitude post}.
\end{align}
Then \eqref{cons of vel magitude ternary} can be written as
\begin{equation}\label{ternary velocities equality}
|\bm{\widetilde{u}}|=|\bm{\widetilde{u}^*}|,
\end{equation}
which is the ternary analog of the binary expression \eqref{binary rel vel equation}.
Defining 
\begin{equation}\label{ternary relative velocities unit}
\bm{\bar{u}}:=\frac{\bm{u}}{|\bm{\widetilde{u}}|},\quad\bm{\bar{u}^*}:=\frac{\bm{u^*}}{|\bm{\widetilde{u}}|},
\end{equation}
equality \eqref{cons of vel magitude ternary} implies 
$\bm{\bar{u}},\bm{\bar{u}^*}\in\mathbb{E}^{2d-1}$,
where 
\begin{equation}\label{ellipsoid}
\mathbb{E}^{2d-1}:=\{(\nu_1,\nu_2)\in\mathbb{R}^{2d}: |\nu_1|^2+|\nu_2|^2+|\nu_1-\nu_2|^2=1\},
\end{equation}
is a $(2d-1)$-dimensional ellipsoid. The vectors $\bm{\bar{u}},\bm{\bar{u}^*}$ are the ternary analogs of the relative velocity direction and the scattering direction of the binary interaction. Because of the assymetry of the ternary interaction they are not unit vectors, they lie on the ellipsoid $\mathbb{E}^{2d-1}$ instead. However, for convenience we will refer to $\bm{\bar{u}},\bm{\bar{u}^*}$ as relative velocity direction and scattering direction respectively.

The collisional formulas \eqref{ternary collision formulas} also imply  
\begin{equation}\label{skewsymmetry ternary}
\bm{\omega}\cdot\bm{\bar{u}^*}=-\bm{\omega}\cdot\bm{\bar{u}},
\end{equation}
which is the ternary analog to specular reflection with respect to the impact directions vector $\bm{\omega}=(\omega_1,\omega_2)\in\mathbb{S}^{2d-1}$.  Indeed,  one has
\begin{align*}
\boldsymbol{\omega} \cdot \boldsymbol{u^*}
	&= \omega_1 \cdot (v_1^* - v^* ) + \omega_2 \cdot (v_2^* - v^*) = \bm{u\cdot\omega}
		 -   \frac{2\bm{u\cdot\omega}}{1+\omega_1\cdot \omega_2} 
		 	\Big( |\omega_1|^2  + \omega_1 \cdot \omega_2 + |\omega_2|^2 \Big) =  - \boldsymbol{\omega} \cdot \boldsymbol{u},
\end{align*}
which is equivalent to \eqref{skewsymmetry ternary} due to \eqref{ternary relative velocities unit}.
 
 The term $B_3$ in the integrand of \eqref{ternary kernel},  depending on the relative velocity $\bm{u}\in\mathbb{R}^{2d}$ and the impact directions vector $\bm{\omega}=(\omega_1,\omega_2)\in\mathbb{S}^{2d-1}$, is the  ternary interaction differential cross-section, which describes the transition probability of ternary interactions. Recalling $|\bm{\widetilde{u}}|$ from \eqref{ternary magnitude pre} and $\bm{\bar{u}}\in\mathbb{E}^{2d-1}$ from \eqref{ternary relative velocities unit}, we assume  $B_3$ takes the form
\begin{equation}\label{form of B_3}
B_3(\bm{u},\bm{\omega})=|\bm{\widetilde{u}}|^{\gamma_3}b_3\left(\bm{\bar{u}}\cdot\bm{\omega},\omega_1\cdot\omega_2\right),\quad\gamma_3\in(-2d+1,1],
\end{equation}
and $b_3:[-1,1]\times[-\frac{1}{2},\frac{1}{2}]\to [0,\infty)$ is the ternary angular transition probability density.   Since $\bm{\omega}=(\omega_1,\omega_2)\in\mathbb{S}^{2d-1}$,  Cauchy-Schwartz inequality and \eqref{ternary magnitude pre} yield
\begin{equation}\label{ternary magnitude pre-2}
|\bm{\bar{u}}\cdot\bm{\omega}|\leq|\bm{\bar{u}}|\cdot|\bm{\omega}|=\frac{|\bm{u}|}{|\bm{\widetilde{u}}|}\leq 1.
\end{equation}

Moreover, for any $\bm{\omega}=(\omega_1,\omega_2)\in\mathbb{S}^{2d-1}$, Cauchy-Schwartz inequality followed by Young's inequality yield
$$|\omega_1\cdot\omega_2|\leq |\omega_1|\cdot|\omega_2|\leq\frac{|\omega_1|^2+|\omega_2|^2}{2}=\frac{1}{2}. $$
Therefore, for any $\bm{\omega}=(\omega_1,\omega_2)\in\mathbb{S}_1^{2d-1}$, the expression $b_3\left(\bm{\bar{u}}\cdot\bm{\omega}, \omega_1\cdot\omega_2 \right)$ is well defined.

 In addition, we assume that $b_3$ satisfies the following properties
\begin{itemize}
\item $b_3:[-1,1]\times[-\frac{1}{2},\frac{1}{2}]\to\mathbb{R} $ is a measurable, non-negative probability density.
\item $b_3$ is even with respect to the first argument i.e.
\begin{equation}\label{ternary micro2}
b_3(-z,w)=b_3(z,w),\quad\forall (z,w)\in[-1,1]\times[-\frac{1}{2},\frac{1}{2}]\, .
\end{equation}
In addition,  due to \eqref{skewsymmetry ternary},  the ternary micro-reversibility condition holds
\begin{equation}\label{ternary micro}
b_3\left(\bm{\bar{u}^*}\cdot\bm{\omega},\omega_1\cdot\omega_2\right)=b_3\left(\bm{\bar{u}}\cdot\bm{\omega}, \omega_1\cdot\omega_2 \right),\quad\forall\bm{\omega}\in\mathbb{S}^{2d-1},\quad\forall v,v_1,v_2\in\mathbb{R}^d\ , 
\end{equation}
and relations \eqref{form of B_3}, \eqref{ternary relative velocities unit} and \eqref{ternary micro} imply the total ternary collision kernel  satisfies
\begin{equation}\label{pre-post  ternary cross}
 B_3(\bm{u^*},\bm{\omega})=B_3(\bm{u},\bm{\omega}),\quad\forall\bm{\omega}\in\mathbb{S}^{2d-1},\quad\forall v,v_1,v_2\in\mathbb{R}^d.
\end{equation}\\
\item The probability density $b_3$ is integrable on $\mathbb{S}^{2d-1}$ i.e. 
\begin{equation}\label{Grad ternary}
\|b_3\|_{L^1(\mathbb{S}^{2d-1})}:=\sup_{\bm{\nu}\in\mathbb{E}_1^{2d-1}}\int_{\mathbb{S}^{2d-1}}b_3(\bm{\nu}\cdot\bm{\omega},\omega_1\cdot\omega_2)\,d\bm{\omega}<\infty.
\end{equation}
\end{itemize}

\begin{remark}
One can see that the ternary operator  introduced in \cite{ours} is a special case of \eqref{form of B_3} for
$$\gamma_3=1,\quad b_3(z,w)=\frac{1}{2}\frac{|z|}{\sqrt{1+w}}.$$
\end{remark}

\begin{remark}\label{remark on constants}
Throughout the paper we assume that at least one of $b_2$, $b_3$ is not trivially zero; however one of the two could be zero. If $b_3=0$, we recover the  classical Boltzmann equation \eqref{intro-binary}, while if $b_2=0$ we recover the ternary Boltzmann equation \eqref{intro-ternary}. As it will become clear, see for instance \eqref{constant wide K intro},  the dependence on the size of $b_2$ and $b_3$ is additive implying that the two collisional operators can be studied separately. 
\end{remark}
\subsubsection{Gain and loss operators} It turns out more convenient to study the more general collisional operators
\begin{equation}\label{binary kernel general}
Q_2(f,g)(t,x,v)=\int_{\mathbb{S}^{d-1}\times\mathbb{R}^d}B_2(u,\omega)\left(f(t,x,v')g(t,x,v_1')-f(t,x,v)g(t,x,v_1)\right)\,d\omega\,dv_1,
\end{equation} 
\begin{equation}\label{ternary kernel general}
\begin{aligned}
Q_3(f,g,h)(t,x,v)=\int_{\mathbb{S}^{2d-1}\times\mathbb{R}^{2d}}B_3( \bm{u},\bm{\omega})\left(f(t,x,v^*)g(t,x,v_1^*)h(t,x,v_2^*)-f(t,x,v)g(t,x,v_1)h(t,x,v_2)\right)&\\
\,d\omega_1\,d\omega_2\,dv_1\,dv_2&,
\end{aligned}
\end{equation}
Due to the  assumptions \eqref{Grad binary}, \eqref{Grad ternary}, the binary-ternary operator $Q_2(f,g)+Q_3(f,g,h)$ can be decomposed into a gain and a loss term as follows
\begin{equation}\label{decomposition to gain loss}
Q_2(f,g)+Q_3(f,g,h)=G(f,g,h)-L(f,g,h),
\end{equation}
 where
\begin{align}
L(f,g,h)&=L_2(f,g)+L_3(f,g,h)\label{L},\\
G(f,g,h)&=G_2(f,g)+G_3(f,g,h).\label{G}
 \end{align}
The binary gain and loss operators $G_2,L_2$ are given respectively by
\begin{align}
G_2(f,g)(t,x,v)&=\int_{\mathbb{S}^{d-1}\times\mathbb{R}^d}B_2(u,\omega)f(t,x,v')g(t,x,v_1')\,d\omega\,dv_1,\label{G_2}\\
L_2(f,g)(t,x,v)&=\int_{\mathbb{S}^{d-1}\times\mathbb{R}^d}B_2(u,\omega)f(t,x,v)g(t,x,v_1)\,d\omega\,dv_1,\label{L2}\\
\end{align}
and are clearly bilinear. The ternary gain and loss operators $L_3,G_3$ are given respectively by
\begin{align}
G_3(f,g,h)(t,x,v)&=\int_{\mathbb{S}^{2d-1}\times\mathbb{R}^{2d}}B_3(\bm{u},\bm{\omega}) f(t,x,v^*)g(t,x,v_1^*)h(t,x,v_2^*)\,d\omega_1\,d\omega_2\,dv_1\,dv_2,\label{G_3}\\
L_3(f,g,h)(t,x,v)&=\int_{\mathbb{S}^{2d-1}\times\mathbb{R}^{2d}}B_3(\bm{u},\bm{\omega})f(t,x,v)g(t,x,v_1)h(t,x,v_2)\,d\omega_1\,d\omega_2\,dv_1\,dv_2,\label{L_3}
\end{align}
and are clearly trilinear.
Notice  the loss term can be factorized as
\begin{equation}\label{L-R}
L(f,g,h)=fR(g,h),
\end{equation}
where $R$ is given by
\begin{equation}\label{def of R}
R(g,h):=R_2(g)+R_3(g,h),
\end{equation}
$R_2$ is the linear operator
\begin{equation}\label{R_2}
R_2(g)(t,x,v):=\int_{\mathbb{S}^{d-1}\times\mathbb{R}^d}B_2(u,\omega)g(t,x,v_1)\,d\omega\,dv_1,
\end{equation}
and
 $R_3$ is the bilinear operator
\begin{equation}\label{R_3}
R_3(g,h)(t,x,v):=\int_{\mathbb{S}^{2d-1}\times\mathbb{R}^{2d}}B_3(\bm{u},\bm{\omega})g(t,x,v_1)h(t,x,v_2)\,d\omega_1\,d\omega_2\,dv_1\,dv_2.
\end{equation}
\subsection{Some notation and the notion of a solution }\label{subsec: notation and solution}
Throughout the paper, the dimension $d\geq 2$, the binary and ternary integrability assumptions  \eqref{Grad binary}, \eqref{Grad ternary} respectively, and the cross-section exponents
\begin{equation}\label{exponents}
\gamma_2\in(-d+1,1],\quad \gamma_3\in (-2d+1,1],
\end{equation}
appearing respectively  in \eqref{form of B_2}, \eqref{form of B_3} will be fixed. Moreover, $C_d$ denotes a general constant depending on the dimension $d$ and can change values.
\subsubsection{Functional spaces}
Let us introduce the functional spaces used in this paper. 
First, in order to point out the dependence in positions and velocities,  we will use the notation
\begin{align}
L^1_{x,v}&:=L^1(\mathbb{R}^d\times\mathbb{R}^d),\label{L1 xv}\\
L^\infty_{x,v}&:=L^\infty(\mathbb{R}^d\times\mathbb{R}^d).\label{L infty xv}
\end{align}
We also define  the sets of space-velocity functions
\begin{align}
F_{x,v}&:=\{f:\mathbb{R}^d\times\mathbb{R}^d\to\overline{\mathbb{R}}, \text{ such that $f$ is measurable}\},\label{functions xv}\\
F_{x,v}^+&:=\{f\in F_{x,v}: f(x,v)\geq 0,\text{ for a.e. $(x,v)\in\mathbb{R}^d\times\mathbb{R}^d$}\},\label{positive functions xv}\\
L^{1,+}_{x,v}&:= L^1_{x,v}\cap F_{x,v}^+.\label{positive functions xv integrable}
\end{align}
In general, for $f,g\in F_{x,v}$, we write $f\geq g$ iff $f(x,v)\geq g(x,v)$ for a.e. $(x,v)\in\mathbb{R}^d\times\mathbb{R}^d$. Same notation will hold for
 equality as well.
 
Given $\alpha,\beta>0$, we define the corresponding (non-normalized) Maxwellian $M_{\alpha,\beta}:\mathbb{R}^{d}\times\mathbb{R}^d\to (0,\infty)$ by
\begin{equation}\label{mawellian}
M_{\alpha,\beta}(x,v):=e^{-\alpha|x|^2-\beta|v|^2}.
\end{equation}
We also define the corresponding Banach space of functions essentially bounded by $M_{\alpha,\beta}$ as
\begin{equation}\label{banach space of maxwellians}
\mathcal{M}_{\alpha,\beta}:=\{f\in F_{x,v}:\|f\|_{\mathcal{M}_{\alpha,\beta}}<\infty\},
\end{equation}
where
$$\|f\|_{\mathcal{M}_{\alpha,\beta}}:=\|fM_{\alpha,\beta}^{-1}\|_{L^\infty_{x,v}}.$$
We will write $f_n\overset{\mathcal{M}_{\alpha,\beta}}\longrightarrow f$ if
\begin{equation}\label{M-convergence}
f_n\overset{\text{a.e.}}\longrightarrow f\quad\text{and}\quad\sup_{n\in\mathbb{N}}\|f_n\|_{\mathcal{M}_{\alpha,\beta}}<\infty.
\end{equation}
 It is clear that if $f_n\overset{\mathcal{M}_{\alpha,\beta}}\longrightarrow f $ then $f_n\in\mathcal{M}_{\alpha,\beta}$ for all $n\in\mathbb{N}$ and $f\in\mathcal{M}_{\alpha,\beta}$. If $k\in\mathbb{N}$ and $f_{1,n}\overset{\mathcal{M}_{\alpha,\beta}}\longrightarrow f_{1}$, $f_{2,n}\overset{\mathcal{M}_{\alpha,\beta}}\longrightarrow f_{2}$,..., $f_{k,n}\overset{\mathcal{M}_{\alpha,\beta}}\longrightarrow f_{k}$, we will write 
$$(f_{1,n},...,f_{k,n})\overset{\mathcal{M}_{\alpha,\beta}}\longrightarrow (f_1,...,f_k).$$
We also define the set of a.e. non-negative functions essentially bounded by $M_{\alpha,\beta}$ as
\begin{equation}\label{positive M}
\mathcal{M}_{\alpha,\beta}^+:=\mathcal{M}_{\alpha,\beta}\cap F_{x,v}^+.
\end{equation}
Given $0<T\leq\infty$, we define the sets of time dependent functions
\begin{align}
\mathcal{F}_T&:=\{f:[0,T)\to F_{x,v}\},\label{functions of time}\\
\mathcal{F}_T^+&:=\{f:[0,T)\to F_{x,v}^+\}\label{positive functions of time},
\end{align}
and given $f,g\in\mathcal{F}_T$, we will write $f\geq g$ iff $f(t)\geq g(t)$ for all $t\in [0,T)$. Same notation will hold for equalities as well.

Finally, we define  the following subsets of functional spaces
\begin{align}
C^0([0,T),L^{1,+}_{x,v})&:=C^0([0,T),L^{1}_{x,v})\cap \mathcal{F}_T^+,\label{continuous positive functions}\\
L^1_{loc}([0,T),L_{x,v}^{1,+})&:=L^1_{loc}([0,T),L_{x,v}^{1})\cap \mathcal{F}_T^+\label{L 1 pos def},\\
L^\infty([0,T),L_{x,v}^{1,+})&:=L^\infty([0,T),L_{x,v}^{1})\cap \mathcal{F}_T^+,\label{L infty pos def}
\end{align}
and given $\alpha,\beta>0$, we define the Banach space of time dependent  functions uniformly essentially bounded by  $M_{\alpha,\beta}$
\begin{equation}\label{L infty M}
L^\infty([0,T),\mathcal{M}_{\alpha,\beta}):=\{f\in\mathcal{F}_T: |||f|||_\infty<\infty\},
\end{equation}
with norm 
 \begin{equation}\label{time maxwellian norm}
 |||f|||_\infty=\sup_{t\in[0,T)}\|f(t)\|_{\mathcal{M}_{\alpha,\beta}}.
 \end{equation}
 Notice that in definition \eqref{L infty M}, the supremum is taken with respect to all $t\in[0,T)$. We also write
\begin{equation}\label{positive L infty M}
L^\infty([0,T),\mathcal{M}_{\alpha,\beta}^+):=L^\infty([0,T),\mathcal{M}_{\alpha,\beta})\cap\mathcal{F}_T^+.
\end{equation}

\subsubsection{Transport operator} We now introduce the transport operator which will be crucial to define mild solutions to \eqref{generalized boltzmann equation}. Let us recall from \eqref{functions xv}-\eqref{positive functions xv} the sets of functions
\begin{align*}
F_{x,v}&:=\{f:\mathbb{R}^d\times\mathbb{R}^d\to\overline{\mathbb{R}}, \text{ such that $f$ is measurable}\},\\
F_{x,v}^+&:=\{f\in F_{x,v}: f(x,v)\geq 0,\text{ for a.e. $(x,v)\in\mathbb{R}^d\times\mathbb{R}^d$}\}.
\end{align*}

Consider a positive time $0<T\leq\infty$ (we can have $T=\infty$) and recall from \eqref{functions of time}-\eqref{positive functions of time} the sets of time dependent functions
\begin{align*}
\mathcal{F}_T&:=\{f:[0,T)\to F_{x,v}\},\\
\mathcal{F}_T^+&:=\{f:[0,T)\to F_{x,v}^+\}.
\end{align*}

Given $f\in\mathcal{F}_T$, we define $f^\#\in\mathcal{F}_T$ by
\begin{equation}\label{definition of transport}
f^\#(t,x,v):=f(t,x+tv,v),
\end{equation}
and $f^{-\#}\in\mathcal{F}_T$ by
\begin{equation*}
f^{-\#}(t,x,v):=f(t,x-tv,v).
\end{equation*}
 Clearly, the operators $\#:\mathcal{F}_T\to\mathcal{F}_T$ and $-\#:\mathcal{F}_T\to\mathcal{F}_T$ are linear and invertible and in particular
\begin{equation*}
(\#)^{-1}=-\#.
\end{equation*}
\begin{remark}\label{remark on measure preserving} Let $f,g\in\mathcal{F}_T$.
Since the maps $(x,v)\to (x+tv,v)$ and $(x,v)\to(x-tv,v)$ are measure-preserving, for all $t\in[0,T)$, we have 
\begin{align*}
f\geq g\Leftrightarrow f^\#\geq g^\#\Leftrightarrow f^{-\#}\geq g^{-\#}.
\end{align*}
In particular 
\begin{equation}\label{preservation of ineq under sharp}
f\in \mathcal{F}_T^+\Leftrightarrow f^\#\in \mathcal{F}_T^+\Leftrightarrow f^{-\#}\in \mathcal{F}_T^+.
\end{equation}
\end{remark}
\begin{remark}\label{isometry remark}
Let $f,g\in\mathcal{F}_T$. Since the maps $(x,v)\to (x+tv,v)$ and $(x,v)\to(x-tv,v)$ are measure-preserving, for all $t\in[0,T)$, we have 
\begin{equation}\label{L1 change of variables}
\|f^\#(t)\|_{L^1_{x,v}}=\|f(t)\|_{L^1_{x,v}}=\|f^{-\#}(t)\|_{L^1_{x,v}},\quad\forall t\in[0,T).
\end{equation}
Relation \eqref{preservation of ineq under sharp}-\eqref{L1 change of variables} and linearity of the transport operator imply 
\begin{equation}\label{equivalence of continuity}
f\in C^0([0,T),L^{1,+}_{x,v})\Leftrightarrow f^\#\in C^0([0,T),L^{1,+}_{x,v})\Leftrightarrow f^{-\#}\in C^0([0,T),L^{1,+}_{x,v}).
\end{equation}
\end{remark}

Throughout the manuscript, we will often define $f^\#\in \mathcal{F}_T$ directly, implying that $f$ is defined by $f:=(f^\#)^{-\#}$.

\subsubsection{Transported gain and loss operators} In order to define mild solutions to \eqref{generalized boltzmann equation} , it is important to understand the action of the transport operator on the gain and loss operators. More specifically, given $f,g,h\in\mathcal{F}_T$, for the gain operators we write
\begin{align}
G^\#_2(f,g)(t,x,v)&:=\left(G_2\left(f,g\right)\right)^\#(t,x,v)=\int_{\mathbb{S}^{d-1}\times\mathbb{R}^d}B_2(u,\omega)f(t,x+tv,v')g(t,x+tv,v_1')\,d\omega\,dv_1,\nonumber\\
G^\#_3(f,g,h)(t,x,v)&:=\left(G_3\left(f,g,h\right)\right)^\#(t,x,v)\nonumber\\
&=\int_{\mathbb{S}^{2d-1}\times\mathbb{R}^{2d}}B_3(\bm{u},\bm{\omega})f(t,x+tv,v^*)g(t,x+tv,v_1^*)h(t,x+tv,v_2^*)\,d\omega_1\,d\omega_2\,dv_1\,dv_2,\nonumber\\
G^\#(f,g,h)(t,x,v)&:=G^\#_2(f,g)(t,x,v)+G^\#_3(f,g,h)(t,x,v),\label{G sharp}
\end{align}
and for the loss operators we write
\begin{align}
L^\#_2(f,g)(t,x,v)&:=\left(L_2\left(f,g\right)\right)^\#(t,x,v)=\int_{\mathbb{S}^{d-1}\times\mathbb{R}^d}B_2(u,\omega)f(t,x+tv,v)g(t,x+tv,v_1)\,d\omega\,dv_1,\nonumber\\
L^\#_3(f,g,h)(t,x,v)&:=\left(L_3\left(f,g,h\right)\right)^\#(t,x,v)\nonumber\\
&=\int_{\mathbb{S}^{2d-1}\times\mathbb{R}^{2d}}B_3(\bm{u},\bm{\omega})f(t,x+tv,v)g(t,x+tv,v_1)h(t,x+tv,v_2)\,d\omega_1\,d\omega_2\,dv_1\,dv_2,\nonumber\\
L^\#(f,g,h)(t,x,v)&:=L^\#_2(f,g)(t,x,v)+L^\#_3(f,g,h)(t,x,v).\label{L sharp}
\end{align}
Under this notation, it is  straightforward to verify that 
\begin{equation}\label{connecting L-R sharp}
\begin{aligned}
L_2^\#(f,g,h)(t)&=f^\#(t)R_2^\#(g)(t),\\
L_3^\#(f,g,h)(t)&=f^\#(t)R_3^\#(g,h)(t),\\
L^\#(f,g,h)(t)&=f^\#(t)R^\#(g,h)(t),
\end{aligned}
\end{equation}
where
\begin{align}
R^\#_2(g)(t,x,v):&=\left(R_2\left(g\right)\right)^\#(t,x,v)=\int_{\mathbb{S}^{d-1}\times\mathbb{R}^d}B_2(u,\omega)g(t,x+tv,v_1)\,d\omega\,dv_1,\nonumber\\
R^\#_3(g,h)(t,x,v):&=\left(R_3\left(g,h\right)\right)^\#(t,x,v)\nonumber\\
&=\int_{\mathbb{S}^{2d-1}\times\mathbb{R}^{2d}}B_3(\bm{u},\bm{\omega})g(t,x+tv,v_1)h(t,x+tv,v_2)\,d\omega_1\,d\omega_2\,dv_1\,dv_2,\nonumber\\
R^\#(g,h)(t,x,v):&=R^\#_2(g)(t,x,v)+R^\#_3(g,h)(t,x,v).\label{R sharp}
\end{align}

\subsubsection{Notion of a mild solution} Using  \eqref{decomposition to gain loss}, the binary-ternary  Boltzmann equation \eqref{generalized boltzmann equation} is written as follows
\begin{equation}\label{generalized boltzmann gain loss form}
\begin{cases}
\partial_tf+v\cdot\nabla_xf=G(f,f,f)-L(f,f,f),\quad (t,x,v)\in (0,\infty)\times\mathbb{R}^d\times\mathbb{R}^d,\\
f(0)=f_0,\quad (x,v)\in\mathbb{R}^d\times\mathbb{R}^d,
\end{cases}
\end{equation}
where the gain term $G(f,f,f)$ and the loss term $L(f,f,f)$ are given by \eqref{G}-\eqref{L} respectively. 

Here is where the importance of the transport operator will become clear. Indeed, using  the chain rule, the initial value problem \eqref{generalized boltzmann gain loss form} can be formally  written as
\begin{equation}\label{boltzmann for transport}
\begin{cases}
\partial_tf^\#+L^\#(f,f,f)=G^\#(f,f,f),\quad (t,x,v)\in (0,\infty)\times\mathbb{R}^d\times\mathbb{R}^d,\\
f^\#(0)=f_0,\quad (x,v)\in\mathbb{R}^d\times\mathbb{R}^d.
\end{cases}
\end{equation}

Motivated by \eqref{boltzmann for transport}, we aim to define solutions of \eqref{generalized boltzmann equation} up to time $0<T\leq\infty$, with respect to a given Maxwellian $M_{\alpha,\beta}$, where $\alpha,\beta>0$.

\begin{definition}\label{boltzmann mild solution} Let  $0<T\leq\infty$, $\alpha,\beta>0$ and $f_0\in\mathcal{M}_{\alpha,\beta}^+$. A mild solution to \eqref{generalized boltzmann equation} in $[0,T)$, with initial data $f_0\in\mathcal{M}_{\alpha,\beta}^+$, is a function  $f\in \mathcal{F}_T^+$ such that 
\begin{enumerate}[(i)]

\item $f^\#\in C^0([0,T),L^{1,+}_{x,v})\cap L^\infty([0,T),\mathcal{M}_{\alpha,\beta}^+)$,
\item $L^\#(f,f,f), G^\#(f,f,f)\in L^\infty([0,T),L^{1,+}_{x,v})$,

\item $f^\#$ is weakly differentiable and satisfies
\begin{equation}\label{mild solution}
\begin{cases}
 \displaystyle\frac{\,df^\#}{\,dt}+L^\#(f,f,f)= G^\#(f,f,f),\\
f^\#(0)=f_0.
\end{cases}
\end{equation}
\end{enumerate}
\end{definition}
\begin{remark} The differential equation of \eqref{mild solution} is interpreted as an equality of distributions since all terms involved belong to $L^1_{loc}([0,T),L^{1,+}_{x,v})$.
\end{remark}
\begin{remark} Remarks \ref{remark on measure preserving}-\ref{isometry remark} imply that a mild solution $f$ to \eqref{generalized boltzmann equation} belongs to $C^0([0,T),L^{1,+}_{x,v})$.
\end{remark}

\subsection{Statement of the main result}\label{subsec: statement}
Now we are ready to state the main result of the paper. 
\begin{theorem}\label{gwp theorem}
Let  $0<T\leq\infty$, $\alpha,\beta>0$.
Then for any initial data $f_0\in\mathcal{M}_{\alpha,\beta}^+$ with
\begin{equation}\label{condition for existence}
\|f_0\|_{\mathcal{M}_{\alpha,\beta}}< \frac{\alpha^{1/2}}{48K_\beta(1+\frac{\alpha^{1/4}}{2\sqrt{6K_\beta}})},
\end{equation}
where
\begin{equation}
\begin{aligned}
K_\beta=C_d\bigg[\|b_2\|_{L^1(\mathbb{S}^{d-1})}(\beta^{-d/2}+\frac{1}{d+\gamma_2-1})+\|b_3\|_{L^1(\mathbb{S}^{2d-1})}(\beta^{-d}+\frac{1}{2d+\gamma_3-1})\bigg]>0,\label{constant wide K intro}
\end{aligned}
\end{equation}
 and $C_d$ is an appropriate constant depending on the dimension $d$,
 the binary-ternary  Boltzmann equation \eqref{generalized boltzmann equation} has a unique mild solution $f$ satisfying the bound
\begin{equation}\label{condition for uniqueness} 
 |||f^\#|||_{\infty}\leq \frac{1-\sqrt{1-48K_\beta\alpha^{-1/2}(1+\frac{\alpha^{1/4}}{2\sqrt{6K_\beta}})\|f_0\|_{\mathcal{M}_{\alpha,\beta}}}}{24K_\beta\alpha^{-1/2}\left(1+\frac{\alpha^{1/4}}{2\sqrt{6K_\beta}}\right)}.
 \end{equation}
 \end{theorem}
 \begin{remark} As we will see, the uniqueness claimed above holds in the class of solutions of \eqref{generalized boltzmann equation} satisfying \eqref{condition for uniqueness}.
 \end{remark}

\begin{remark} 
According to the assumptions on $b_2,b_3$ made in Remark \ref{remark on constants}, Theorem \ref{gwp theorem} applies as well to the end point cases where either $b_2=0$ or $b_3=0$ (but not both). In the  case $b_3=0$, one recovers the solution of the classical Boltzmann equation \eqref{intro-binary} constructed in \cite{illner-shinbrot}, while in the case $b_2=0$, one obtains well-posedness of the ternary Boltzmann equation \eqref{intro-ternary}, introduced in \cite{ours}.  
\end{remark}

\section{Properties of the transported gain and loss operators}\label{sec: properties of gain and loss}
In this section, we investigate properties of the transported gain and loss operators which will be important for proving global well-posedness of \eqref{generalized boltzmann equation}. 

\subsection{Monotonicity and $L^1$-norms}
As we will see, the transported gain and loss operators are monotone increasing when acting on non-negative functions. These monotonicity properties will allow us to construct monotone sequences of supersolutions and subsolutions to \eqref{generalized boltzmann equation}. Moreover, we show that the $L^1$-norm of the gain is equal to the $L^1$-norm of the loss.  This equality will allow us to reduce estimates on the norm of the gain term to estimating the norm of the loss term.
In the following, saying that an operator is bilinear/trilinear, we mean it is linear in each argument, and saying it is monotone increasing, we mean it is increasing in each argument. 
 \begin{proposition}\label{monotonicity proposition sharp} Let $0<T\leq\infty$. Then the following hold
 \begin{enumerate}[(i)]
 \item $R_2^\#:\mathcal{F}_T^+\to\mathcal{F}_T^+$ is linear and monotone increasing.
 \item $L_2^\#,G_2^\#,R_3^\#:\mathcal{F}_T^+\times\mathcal{F}_T^+\to\mathcal{F}_T^+$ are bilinear and monotone increasing.
 \item $L_3^\#, G_3^\#:\mathcal{F}_T^+\times\mathcal{F}_T^+\times\mathcal{F}_T^+\to\mathcal{F}_T^+$ are trilinear and monotone increasing.
 \item $L^\#,G^\#:\mathcal{F}_T^+\times\mathcal{F}_T^+\times\mathcal{F}_T^+\to\mathcal{F}_T^+$ and $R^\#:\mathcal{F}_T^+\times \mathcal{F}_T^+\to \mathcal{F}_T^+$ are monotone increasing.
 \item For any $f,g,h\in\mathcal{F}_T^+$, the following identities hold
 \begin{equation}\label{G leq L sharp}
 \begin{aligned}
\|G_2^\#(f,g)(t)\|_{L^1_{x,v}}&= \|L_2^\#(f,g)(t)\|_{L^1_{x,v}},\quad\forall t\in[0,T),\\
\|G_3^\#(f,g,h)(t)\|_{L^1_{x,v}}&= \|L_3^\#(f,g,h)(t)\|_{L^1_{x,v}},\quad\forall t\in[0,T),\\
\|G^\#(f,g,h)(t)\|_{L^1_{x,v}}&= \|L^\#(f,g,h)(t)\|_{L^1_{x,v}},\quad\forall t\in[0,T).
 \end{aligned}
 \end{equation}
 \end{enumerate}
 \end{proposition}
 \begin{proof}
 Parts \textit{(i)-(iv)} are immediate by linearity of the integral, positivity of the functions considered and relation \eqref{preservation of ineq under sharp}.
 
 Let us now prove {\em (v)}.  We first prove \eqref{G leq L sharp} for the binary case. By \eqref{L1 change of variables}, we have
$$\|G_2^\#(f,g)(t)\|_{L^1_{x,v}}= \|G_2(f,g)(t)\|_{L^1_{x,v}},\quad \|L_2^\#(f,g)(t)\|_{L^1_{x,v}}= \|L_2(f,g)(t)\|_{L^1_{x,v}},\quad\forall t\in [0,T).$$
Therefore, for any $t\in[0,T)$, using \eqref{pre-post bin cross} and involutionary substitution $(v',v_1')\to (v,v_1)$, we obtain
\begin{align*}
\|G_2^\#(f,g)(t)\|_{L^1_{x,v}}&=\|G_2(f,g)(t)\|_{L^1_{x,v}}\\
&=\int_{\mathbb{R}^{3d}\times\mathbb{S}^{d-1}}B_2(u,\omega)f(t,x,v')g(t,x,v_1')\,d\omega\,dv_1\,dv\,dx\nonumber\\
&=\int_{\mathbb{R}^{3d}\times\mathbb{S}^{d-1}}B_2(u',\omega)f(t,x,v')g(t,x,v_1')\,d\omega\,dv_1\,dv\,dx\\
&=\int_{\mathbb{R}^{3d}\times\mathbb{S}^{d-1}}B_2(u,\omega)f(t,x,v)g(t,x,v_1)\,d\omega\,dv_1\,dv\,dx\\
&=\|L_2(f,g)(t)\|_{L^1_{x,v}}=\|L_2^\#(f,g)(t)\|_{L^1_{x,v}}\nonumber.
\end{align*}
We now prove \eqref{G leq L sharp} for the ternary case. By \eqref{L1 change of variables}, we have
$$\|G_3^\#(f,g,h)(t)\|_{L^1_{x,v}}= \|G_3(f,g,h)(t)\|_{L^1_{x,v}},\quad \|L_3^\#(f,g,h)(t)\|_{L^1_{x,v}}= \|L_3(f,g,h)(t)\|_{L^1_{x,v}},\quad\forall t\in [0,T).$$
Therefore, for any $t\in[0,T)$, using \eqref{pre-post  ternary cross} and the involutionary substitution $(v^*,v_1^*,v_2^*)\to (v,v_1,v_2)$, we obtain
\begin{align*}
\|G_3^\#(f,g,h)(t)\|_{L^1_{x,v}}&=\|G_3(f,g,h)(t)\|_{L^1_{x,v}}\\
&=\int_{\mathbb{R}^{4d}\times\mathbb{S}^{2d-1}}B_3(\bm{u},\bm{\omega})f(t,x,v^*)g(t,x,v_1^*)h(t,x,v_2^*)\,d\omega_1\,d\omega_2\,dv_1\,dv_2\,dv\,dx\nonumber\\
&=\int_{\mathbb{R}^{4d}\times\mathbb{S}^{2d-1}}B_3(\bm{u^*},\bm{\omega})f(t,x,v^*)g(t,x,v_1^*)h(t,x,v_2^*)\,d\omega_1\,d\omega_2\,dv_1\,dv_2\,dv\,dx\nonumber\\
&=\int_{\mathbb{R}^{4d}\times\mathbb{S}^{2d-1}}B_3(\bm{u},\bm{\omega})f(t,x,v)g(t,x,v_1)h(t,x,v_2)\,d\omega_1\,d\omega_2\,dv_1\,dv_2\,dv\,dx\\
&=\|L_3(f,g,h)(t)\|_{L^1_{x,v}}=\|L_3^\#(f,g,h)(t)\|_{L^1_{x,v}}.
\end{align*}
We finally  prove \eqref{G leq L sharp} for the mixed case. By positivity,  for any $t\in[0,T)$, we have
\begin{align*}
\|G^\#(f,g,h)(t)\|_{L^1_{x,v}}&=\|G_2^\#(f,g)(t)+G_3^\#(f,g,h)(t)\|_{L^1_{x,v}}=\|G_2^\#(f,g)(t)\|_{L^1_{x,v}}+\|G_3^\#(f,g,h)(t)\|_{L^1_{x,v}},\\
\|L^\#(f,g,h)(t)\|_{L^1_{x,v}}&=\|L_2^\#(f,g)(t)+L_3^\#(f,g,h)(t)\|_{L^1_{x,v}}=\|L_2^\#(f,g)(t)\|_{L^1_{x,v}}+\|L_3^\#(f,g,h)(t)\|_{L^1_{x,v}}.
\end{align*}
Equality \eqref{G leq L sharp} for the mixed case  immediately follows from the corresponding binary  and ternary equalities.
 \end{proof}

 \subsection{Convolution estimates}
We now present a general convolution-type result, which will  be essential for the control of the binary and the ternary collisional operators. These estimates will be of fundamental importance in the proof of the $L^\infty L^1$ estimates (see Subsection \ref{L_inf L1 estimates}) and the global estimate on the time average of the transported gain and loss operators appearing in Proposition \ref{bounds on operators proposition}, which in turn will be crucial for the proof of global well-posedness of \eqref{generalized boltzmann equation}. For the binary case one can find similar convolution estimates in \cite{kaniel-shinbrot, illner-shinbrot, alonso-gamba}. Here, our contribution is the derivation of these estimates for the ternary case, since this is the first time global well-posedness is studied for such a ternary correction of the Boltzmann equation. The estimates of the ternary term illustrate that consideration of softer potentials is allowed for the ternary collisional operator.
\begin{lemma}\label{convolution lemma} Let $\beta>0$, $q_2\in (-d,1]$ and $q_3\in (-2d,1]$. Then the following  hold 
\begin{enumerate}[(i)]
\item  For any $v\in\mathbb{R}^d$, we have
\begin{equation}\label{conv for binary}
\int_{\mathbb{R}^d}|u|^{q_2} e^{-\beta|v_1|^2}\,dv_1\leq \widetilde{K}_{\beta,q_2}^2(1+|v|^{q_2^+}),
\end{equation}
where $u=v_1-v$, $q_2^+:=\max\{0,q_2\}$, $\widetilde{K}_{\beta,q_2}^2$ is given by
\begin{equation}\label{constant wide K_2}
\widetilde{K}_{\beta,q_2}^2=C_d\bigg[(1+\beta^{-d/2}+\beta^{-\frac{d+1}{2}})\mathds{1}_{q_2>0}(q_2)+(\beta^{-d/2}+\frac{1}{d+q_2})\mathds{1}_{q_2\leq 0}(q_2)\bigg],
\end{equation}
and $C_d$ is an appropriate constant depending on the dimension $d$.
\vspace{0.3cm}
\item For any $v\in\mathbb{R}^d$, we have
\begin{equation}\label{conv for ternary}
\int_{\mathbb{R}^{2d}}|\bm{\widetilde{u}}|^{q_3} e^{-\beta(|v_1|^2+|v_2|^2)}\,dv_1\,dv_2\leq\widetilde{K}_{\beta,q_3}^3(1+|v|^{q_3^+}),
\end{equation}
where $|\bm{\widetilde{u}}|$ is given by \eqref{ternary magnitude pre}, $q_3^+:=\max\{0,q_3\}$, $\widetilde{K}_{\beta,q_3}^3$ is given by
\end{enumerate}
\begin{equation}\label{constant wide K_3}
\widetilde{K}_{\beta,q_3}^3=C_d\bigg[(1+\beta^{-d}+\beta^{-\frac{2d+1}{2}})\mathds{1}_{q_3>0}(q_3)+(\beta^{-d}+\frac{1}{2d+q_3})\mathds{1}_{q_3\leq 0}(q_3)\bigg],
\end{equation}
and $C_d$ is an appropriate constant depending on the dimension $d$.
\end{lemma}

\begin{proof}  We will rely on the elementary estimate
\begin{equation}\label{estimate on gaussian}
\int_{\mathbb{R}^d}e^{-\beta|v_1|^2}\,dv_1\leq C_d\beta^{-d/2},
\end{equation}
and, given $q\in (0,1]$, on the estimate
\begin{align}
\int_{\mathbb{R}^d}|v_1|^{q}e^{-\beta|v_1|^2}\,dv_1&\leq |B_1^d|+\int_{|v_1|>1}|v_1|^{q}e^{-\beta|v_1|^2}\,dv_1\nonumber\\
&\leq |B_1^d|+\int_{|v_1|>1}|v_1|e^{-\beta|v_1|^2}\,dv_1\nonumber\\
&\leq C_d(1+\beta^{-\frac{d+1}{2}}),\label{estimate on polyn}
\end{align}
where  $|B_1^d|$ denotes the volume of the $d$-dimensional unit ball. 

\textit{(i)} We take separate cases for $q_2\in(-d,1]$
\begin{itemize}
\item $q_2\in (0,1]$: Since $q_2\in(0,1]$, we have 
$$|u|^{q_2}=|v-v_1|^{q_2}\leq (|v|+|v_1|)^{q_2}\leq  |v|^{q_2}+|v_1|^{q_2}.$$
Therefore
\begin{align}
\int_{\mathbb{R}^d}|u|^{q_2} e^{-\beta|v_1|^2}\,dv_1
& \leq  \int_{\mathbb{R}^d} (|v|^{q_2}+|v_1|^{q_2})e^{-\beta|v_1|^2}\,dv_1\nonumber\\
& \leq C_{d}(1+\beta^{-d/2}+\beta^{-\frac{d+1}{2}})(1+|v|^{q_2}),\label{conv binary pos}
\end{align}
 where to obtain \eqref{conv binary pos}, we use the estimates \eqref{estimate on gaussian}-\eqref{estimate on polyn} for $q=q_2$.
 
\item $q_2\in (-d,0]$: Since $q_2\leq 0$, estimate  \eqref{estimate on gaussian} implies
\begin{align}
\int_{\mathbb{R}^d}|v-v_1|^{q_2} e^{-\beta|v_1|^2}\,dv_1&\leq\int_{|v-v_1|>1}e^{-\beta|v_1|^2}\,dv_1+\int_{|v-v_1|<1}|v-v_1|^{q_2} \,dv_1\nonumber\\
&= C_d\beta^{-d/2}+\int_{|y|<1}|y|^{q_2}\,dy\nonumber\\
&= C_d\beta^{-d/2}+C_d\int_0^1 r^{d-1+q_2}\,dr\nonumber\\
&= C_{d}\left(\beta^{-d/2}+\frac{1}{d+q_2}\right),\label{conv binary neg}
\end{align}
since we have assumed $q_2>-d$.
\end{itemize}

\textit{(ii)} We take separate cases for $q_3\in(-2d,1]$
\begin{itemize}
\item $q_3\in (0,1]$: Since $q_3\in (0,1]$, we have 

\begin{align*}
|\bm{\widetilde{u}}|^{q_3}&=\left(|v-v_1|^2+|v-v_2|^2+|v_1-v_2|^2\right)^{q_3/2}\\
&\leq 2^{q_3}(|v|^2+|v_1|^2+|v_2|^2)^{q_3/2}\\
\leq 2(|v|^{q_3}+|v_1|^{q_3}+|v_2|^{q_3}).
\end{align*}

Therefore, Fubini's Theorem and estimates \eqref{estimate on gaussian}-\eqref{estimate on polyn} applied for $q=q_3$ imply
\begin{align}
\int_{\mathbb{R}^{2d}}|\bm{\widetilde{u}}|^{q_3} e^{-\beta(|v_1|^2+|v_2|^2)}\,dv_1\,dv_2&\leq 2\int_{\mathbb{R}^{2d}} (|v|^{q_3}+|v_1|^{q_3}+|v_2|^{q_3})e^{-\beta(|v_1|^2+|v_2|^2)}\,dv_1\,dv_2\nonumber\\
&\leq C_{d}(1+\beta^{-d}+\beta^{-\frac{2d+1}{2}})(1+|v|^{q_3}).\label{conv tern pos}
\end{align}
\item $q_3\in (-2d,0]$: 
Recalling \eqref{ternary magnitude pre} and using the fact that $q_3\leq 0$,  Fubini's Theorem and  estimates \eqref{estimate on gaussian}-\eqref{estimate on polyn} imply
\begin{align}
\int_{\mathbb{R}^{2d}}&|\bm{\widetilde{u}}|^{q_3} e^{-\beta(|v_1|^2+|v_2|^2)}\,dv_1\,dv_2\leq \int_{\mathbb{R}^{2d}}|\bm{u}|^{q_3} e^{-\beta(|v_1|^2+|v_2|^2)}\,dv_1\,dv_2\nonumber\\
&\leq\int_{|\bm{u}|>1} e^{-\beta(|v_1|^2+|v_2|^2)}\,dv_1\,dv_2+\int_{|\bm{u}|<1}|\bm{u}|^{q_3} \,dv_1\,dv_2\nonumber\\
&\leq C_d\beta^{-d}+\int_{|\bm{u}|<1}|\bm{u}|^{q_3} \,dv_1\,dv_2\nonumber\\
&=C_d\beta^{-d}+\int_{|\bm{y}|<1}|\bm{y}|^{q_3}\,d\bm{y}\nonumber\\
&\leq C_d\beta^{-d}+C_d\int_0^{1}r^{2d-1+q_3}\,dr\nonumber\\
&= C_d\left(\beta^{-d}+\frac{1}{2d+q_3}\right),\label{conv tern neg}
\end{align}
since we have assumed $q_3>-2d.$ 
\end{itemize}
Combining \eqref{conv binary pos}-\eqref{conv tern pos} and \eqref{conv tern neg}, we obtain \eqref{conv for binary}-\eqref{conv for ternary}.
\end{proof}

\subsection{$L^\infty L^1$ estimates}\label{L_inf L1 estimates}
Here we prove uniform in time, space-velocity $L^1$ estimates on the transported gain and loss operators. These estimates will be of fundamental importance for the convergence of the iteration to the global solution. As we will see, the ternary collisional operator introduces some asymmetry which is not present in the binary case.  For this reason, when we use Lemma \ref{convolution lemma}, we first obtain estimates in asymmetric form (see Lemma \ref{bound on R tilde lemma}).   However, we will need a symmetric version of this estimate which we derive in Proposition \ref{permutation lemma}. To achieve that, we   crucially rely on  properties of the ternary interactions.

Recall from \eqref{exponents} the fixed cross-section exponents $\gamma_2\in (-d+1,1]$ and $\gamma_3\in (-2d+1,1]$. For convenience, we define the function 
\begin{equation}\label{def of p}
p_{\gamma_2,\gamma_3}(v)=1+|v|^{\gamma_2^+}+|v|^{\gamma_3^+}.
\end{equation}
Notice that, given $\alpha>0$, $\beta>0$, we have
\begin{equation}\label{product in L1}
p_{\gamma_2,\gamma_3}M_{\alpha,\beta}\in L^1_{x,v}.
\end{equation}
Using Lemma \ref{convolution lemma} for $q_2=\gamma_2$ and $q_3=\gamma_3$, we obtain some  the assymetric estimates mentioned above.
\begin{lemma}\label{bound on R tilde lemma}Let $0<T\leq\infty$ and $\alpha,\beta>0$. Then there is a constant $C_\beta>0$ such that the following hold
\begin{enumerate}[(i)]
\item For any $g,h\in\mathcal{F}_T^+$, with $g^\#,h^\#\in L^\infty([0,T),\mathcal{M}_{\alpha,\beta}^+)$, and any $t\in[0,T)$, we have
\begin{align}
0\leq R_2^\#(g)(t)&\leq C_\beta|||g^\#|||_\infty p_{\gamma_2,\gamma_3},\label{bound R2}\\
0\leq R_3^\#(g,h)(t)&\leq C_\beta|||g^\#|||_\infty |||h^\#|||_\infty p_{\gamma_2,\gamma_3},\label{bound R3}\\
0\leq R^\#(g,h)(t)&\leq C_\beta|||g^\#|||_{\infty}(1+|||h^\#|||_{\infty})p_{\gamma_2,\gamma_3}.\label{bound R}
\end{align}

\item For any $f,g,h\in\mathcal{F}_T^+$, with $f^\#,g^\#,h^\#\in L^\infty([0,T),\mathcal{M}_{\alpha,\beta}^+)$, and $t\in[0,T)$, we have 
\begin{align}
\|L_2^\#(f,g)(t)\|_{L^1_{x,v}},\text{ }\|G_2^\#(f,g)(t)\|_{L^1_{x,v}}&\leq C_\beta|||g^\#|||_{\infty}\|f^\#(t)p_{\gamma_2,\gamma_3}\|_{L^1_{x,v}},\label{bound L2}\\
\|L_3^\#(f,g,h)(t)\|_{L^1_{x,v}},\text{ }\|G_3^\#(f,g,h)(t)\|_{L^1_{x,v}}&\leq C_\beta |||g^\#|||_{\infty}|||h^\#|||_{\infty}\|f^\#(t)p_{\gamma_2,\gamma_3}\|_{L^1_{x,v}},\label{bound L3}\\
\|L^\#(f,g,h)(t)\|_{L^1_{x,v}},\text{ }\|G^\#(f,g,h)(t)\|_{L^1_{x,v}}&\leq C_\beta|||g^\#|||_{\infty}(1+|||h^\#|||_{\infty})\|f^\#(t)p_{\gamma_2,\gamma_3}\|_{L^1_{x,v}}.\label{bound L}
\end{align} 
 Moreover,
 \begin{equation}\label{L in L infty}
  L^\#(f,g,h), \text{ }G^\#(f,g,h)\in L^\infty([0,T),L^{1,+}_{x,v}).
  \end{equation}
\end{enumerate}
\end{lemma}
\begin{proof}
We prove each claim separately.

\textit{Proof of \textit{(i)}}: Positivity follows immediately by the monotonicity of $R_2^\#,R_3^\#,R^\#$ on $\mathcal{F}_T^+$ (see  Proposition \ref{monotonicity proposition sharp}). Since $g^\#,h^\#\in L^\infty([0,T),\mathcal{M_{\alpha,\beta}^+})$,  for any $t\in[0,T)$, we have
\begin{equation}\label{bound on R only with velocities}
\begin{aligned}
0&\leq g(t,x,v)\leq |||g^\#|||_{\infty}e^{-\alpha |x-tv|^2-\beta|v|^2},\quad\text{for a.e. $(x,v)\in\mathbb{R}^d\times\mathbb{R}^d$},\\
0&\leq h(t,x,v)\leq |||h^\#|||_{\infty}e^{-\alpha |x-tv|^2-\beta|v|^2},\quad\text{for a.e. $(x,v)\in\mathbb{R}^d\times\mathbb{R}^d$}.
\end{aligned}
\end{equation}
Recalling the fact that $R^\#(g,h)=R_2^\#(g)+R_3^\#(g,h)$,  it suffices to prove the estimates \eqref{bound R2}-\eqref{bound R3}.

Let us first prove \eqref{bound R2}. For a.e. $(x,v)\in\mathbb{R}^{2d}$, estimate  \eqref{bound on R only with velocities} and  part \textit{(i)} of Lemma \ref{convolution lemma}, applied for $q_2=\gamma_2$ and $q_3=\gamma_3$, imply
\begin{align}
R_2(g)(t,x,v)&\leq \|b_2\|_{L^1(\mathbb{S}^{d-1})}\int_{\mathbb{R}^d}|u|^{\gamma_2}g(t,x,v_1)\,dv_1\nonumber\\
&\leq \|b_2\|_{L^1(\mathbb{S}^{d-1})}|||g^\#|||_{\infty}\int_{\mathbb{R}^d} |u|^{\gamma_2}e^{-\beta |v_1|^2}\,dv_1\nonumber\\
&\leq C_\beta |||g^\#|||_{\infty}(1+|v|^{\gamma_2^+})\nonumber\\
&\leq C_\beta|||g^\#|||_{\infty}p_{\gamma_2,\gamma_3}(v)\label{final bound R_2 binary}.
\end{align}
Since the right hand side of \eqref{final bound R_2 binary} does not depend on $x$, we obtain \eqref{bound R2}.

Let us now prove \eqref{bound R3}.  For a.e. $(x,v)\in\mathbb{R}^{2d}$, estimate  \eqref{bound on R only with velocities} and  part \textit{(ii)} of Lemma \ref{convolution lemma}, applied for $q_2=\gamma_2$ and $q_3=\gamma_3$, imply
\begin{align}
R_3(g,h)(t,x,v)&\leq \|b_3\|_{L^1(\mathbb{S}^{2d-1})}\int_{\mathbb{R}^{2d}}|\bm{\widetilde{u}}|^{\gamma_3}g(t,x,v_1)h(t,x,v_2)\,dv_1\,dv_2\nonumber\\
&\leq \|b_3\|_{L^1(\mathbb{S}^{2d-1})}|||g^\#|||_{\infty}|||h^\#|||_{\infty}\int_{\mathbb{R}^{2d}}|\bm{\widetilde{u}}|^{\gamma_3}e^{-\beta (|v_1|^2+|v_2|^2)}\,dv_1\,dv_2\nonumber\\
&\leq C_\beta|||g^\#|||_{\infty}|||h^\#|||_{\infty}(1+|v|^{\gamma_3^+})\nonumber\\
&\leq C_\beta|||g^\#|||_{\infty}|||h^\#|||_{\infty}p_{\gamma_2,\gamma_3}(v).\label{final bound R_3 ternary}
\end{align}
 Since the right hand side of \eqref{final bound R_3 ternary} does not depend on $x$, we obtain \eqref{bound R3}. 
 
 Estimate \eqref{bound R} follows by the fact that $R^\#(g,h)=R_2^\#(g)+R_3^\#(g,h)$.

\textit{Proof of \textit{(ii)}}:  We first prove the claim for the loss operators. Positivity follows immediately from the monotonicity of $L_2^\#,L_3^\#,L^\#$ on $\mathcal{F}_T^+$. Estimates \eqref{bound L2}-\eqref{bound L} follow directly from \eqref{connecting L-R sharp} and part \textit{(i)}. Moreover, estimate \eqref{bound L} implies \eqref{L in L infty}
since  $f^\#,g^\#,h^\#\in L^\infty([0,T),\mathcal{M}_{\alpha,\beta}^+)$ and $p_{\gamma_2,\gamma_3}M_{\alpha,\beta}\in L^1_{x,v}$ by \eqref{product in L1}.

For the gain operators,  positivity   follows immediately from the monotonicity of $G_2^\#,G_3^\#,G^\#$ on $\mathcal{F}_T^+$. Estimates \eqref{bound L2}-\eqref{bound L} and \eqref{L in L infty} come from \eqref{G leq L sharp}  and the estimates for the loss operators.
\end{proof}

Notice that bounds \eqref{bound L2}-\eqref{bound L} are only with respect to the first argument $f$. Although this is not an issue in the binary case where the gain and loss collisional operators are symmetric with respect to the inputs in the $L^1$-norm, this is not the case  for the ternary operators. In order to treat this assymetry, we need to derive  estimates with respect to all three inputs of the ternary gain and loss  collisional operators. This is achieved in the following result
\begin{proposition}\label{permutation lemma} Let $0<T\leq\infty$ and $\alpha,\beta>0$. Consider $f_1,f_2,f_3\in\mathcal{F}_T^+$ with $f_1^\#,f_2^\#,f_3^\#\in L^\infty([0,T),\mathcal{M}_{\alpha,\beta}^+)$. Then, there is a constant $C_\beta>0$ such that, for any permutation $\pi:\{1,2,3\}\to\{1,2,3\}$,  the following  estimates hold for any $t\in[0,T)$
\begin{align}
\|L_2^\#(f_1,f_2)(t)\|_{L^1_{x,v}},\text{ }\|G_2^\#(f_1,f_2)(t)\|_{L^1_{x,v}}&\leq C_\beta|||f_{\pi_1}^\#|||_{\infty}\|f^\#_{\pi_2}(t)p_{\gamma_2,\gamma_3}\|_{L^1_{x,v}}\label{perm estimate L2},\\
\|L_3^\#(f_1,f_2,f_3)(t)\|_{L^1_{x,v}},\text{ }\|G_3^\#(f_1,f_2,f_3)(t)\|_{L^1_{x,v}}&\leq C_\beta|||f_{\pi_1}^\#|||_{\infty}|||f_{\pi_2}^\#|||_{\infty}\|f^\#_{\pi_3}(t)p_{\gamma_2,\gamma_3}\|_{L^1_{x,v}},\label{perm estimate L3}\\
\|L^\#(f_1,f_2,f_3)(t)\|_{L^1_{x,v}},\text{ }\|G^\#(f_1,f_2,f_3)(t)\|_{L^1_{x,v}}&\leq C_\beta|||f_{\pi_1}^\#|||_{\infty}(1+|||f_{\pi_2}^\#|||_{\infty})\|f^\#_{\pi_3}(t)p_{\gamma_2,\gamma_3}\|_{L^1_{x,v}}.\label{perm estimate L}
\end{align}
\end{proposition}
\begin{proof}
By \eqref{G leq L sharp},  triangle inequality and part \textit{(ii)} of Lemma \ref{bound on R tilde lemma}, the proof of \eqref{perm estimate L2}-\eqref{perm estimate L} for the loss term reduces to showing the following estimates
\begin{align}
\|L_2^\#(f_1,f_2)(t)\|_{L^1_{x,v}}&\leq C_\beta|||f_1^\#|||_{\infty}\|f_2^\#p_{\gamma_2,\gamma_3}\|_{L^1_{x,v}}\label{perm for L2},\\
\|L_3^\#(f_1,f_2,f_3)(t)\|_{L^1_{x,v}}&\leq C_\beta|||f_1^\#|||_{\infty}|||f_3^\#|||_{\infty}\|f_2^\#p_{\gamma_2,\gamma_3}\|_{L^1_{x,v}}\label{perm for L3 g},\\
\|L_3^\#(f_1,f_2,f_3)(t)\|_{L^1_{x,v}}&\leq C_\beta|||f_1^\#|||_{\infty}|||f_2^\#|||_{\infty}\|f_3^\#p_{\gamma_2,\gamma_3}\|_{L^1_{x,v}}.\label{perm for L3 h}
\end{align} 
$\bullet$ Proof of \eqref{perm for L2}:  Performing the involutionary change of variables $(v,v_1)\to (v_1,v)$ and using \eqref{binary micro2}, for any $t\in[0,T)$, we have
$$\|L_2(f_1,f_2)(t)\|_{L^1_{x,v}}=\|L_2(f_2,f_1)(t)\|_{L^1_{x,v}}\Rightarrow \|L_2^\#(f_1,f_2)(t)\|_{L^1_{x,v}}=\|L_2^\#(f_2,f_1)(t)\|_{L^1_{x,v}}.$$
 The claim comes from part \textit{(ii)} of Lemma \ref{bound on R tilde lemma}.

$\bullet$ Proof of \eqref{perm for L3 g}: Here the proof is subtler because the inner product $\bm{\bar{u}}\cdot\bm{\omega}$ is not symmetric upon renaming the velocities.
However, we will strongly rely on the fact that the expression 
$$|\bm{\widetilde{u}}|^2=|v-v_1|^2+|v-v_2|^2+|v_1-v_2|^2,$$
given in \eqref{ternary magnitude pre} is symmetric with respect to the inputs $v,v_1,v_2$.

Since $f^\#_1,f^\#_3\in L^\infty([0,T),\mathcal{M_{\alpha,\beta}^+})$,  for any $t\in[0,T)$ and a.e. $(x,v)\in\mathbb{R}^d\times\mathbb{R}^d$, we have
\begin{equation}\label{bound on R only with velocities f,h}
\begin{aligned}
0&\leq f_i(t,x,v)\leq  |||f_i^\#|||_{\infty}e^{-\alpha |x-tv|^2-\beta|v|^2}\leq  |||f_i^\#|||_{\infty}e^{-\beta|v|^2},\quad\forall i\in\{1,3\}.\\
\end{aligned}
\end{equation}
Using \eqref{L1 change of variables},  the change of variables $(v,v_1)\to (v_1,v)$, bound \eqref{bound on R only with velocities f,h},  part \textit{(ii)} of Lemma \ref{convolution lemma}, and the fact that $p_{\gamma_2,\gamma_3}$ is invariant in space, we obtain
\begin{align}
\|&L_3^\#(f_1,f_2,f_3)(t)\|_{L^1_{x,v}}=\|L_3(f_1,f_2,f_3)\|_{L^1_{x,v}}\nonumber\\
&\leq \|b_3\|_{L^1(\mathbb{S}^{2d-1})} \int_{\mathbb{R}^{4d}}|\bm{\widetilde{u}}|^{\gamma_3}|f_1(t,x,v)||f_2(t,x,v_1)||f_3(t,x,v_2)|\,dv_1\,dv_2\,dv\,dx\nonumber\\
&= \|b_3\|_{L^1(\mathbb{S}^{2d-1})}\int_{\mathbb{R}^{4d}}(|v-v_1|^2+|v-v_2|^2+|v_1-v_2|^2)^{\gamma_3/2}|f_1(t,x,v)||f_2(t,x,v_1)||f_3(t,x,v_2)|\,dv_1\,dv_2\,dv\,dx\nonumber\\
&=\|b_3\|_{L^1(\mathbb{S}^{2d-1})}\int_{\mathbb{R}^{4d}}(|v-v_1|^2+|v-v_2|^2+|v_1-v_2|^2)^{\gamma_3/2}|f_2(t,x,v)|f_1(t,x,v_1)||f_3(t,x,v_2)|\,dv_1\,dv_2\,dv\,dx\nonumber\\
&\leq \|b_3\|_{L^1(\mathbb{S}^{2d-1})} \int_{\mathbb{R}^{4d}}|\bm{\widetilde{u}}|^{\gamma_3}|f_2(t,x,v)||f_1(t,x,v_1)||f_3(t,x,v_2)|\,dv_1\,dv_2\,dv\,dx\nonumber\\
&=\|b_3\|_{L^1(\mathbb{S}^{2d-1})}\int_{\mathbb{R}^d\times\mathbb{R}^d}|f_2(t,x,v)|\int_{\mathbb{R}^{2d}}|\bm{\widetilde{u}}|^{\gamma_3}|f_1(t,x,v_1)||f_3(t,x,v_2)|\,dv_1\,dv_2\,dv\,dx\nonumber\\
&\leq \|b_3\|_{L^1(\mathbb{S}^{2d-1})}|||f_1^\#|||_{\infty}|||f_3^\#|||_{\infty} \int_{\mathbb{R}^d\times\mathbb{R}^d}|f_2(t,x,v)|\int_{\mathbb{R}^{2d}}|\bm{\widetilde{u}}|^{\gamma_3}e^{-\beta(|v_1|^2+|v_2|^2)}\,dv_1\,dv_2\,dv\,dx\nonumber\\
&\leq  C_\beta |||f_1^\#|||_{\infty}|||f_3^\#|||_{\infty}\|f_2(t)p_{\gamma_2,\gamma_3}\|_{L^1_{x,v}}\nonumber\\
&=  C_\beta |||f_1^\#|||_{\infty}|||f_3^\#|||_{\infty}\|(f_2(t)p_{\gamma_2,\gamma_3})^\#\|_{L^1_{x,v}}\nonumber\\
&= C_\beta |||f_1^\#|||_{\infty}|||f_3^\#|||_{\infty}\|f_2^\#(t) p_{\gamma_2,\gamma_3}\|_{L^1_{x,v}}.\nonumber
\end{align}

$\bullet$ Proof of \eqref{perm for L3 h}: Follows in a similar way to the proof of \eqref{perm for L3 g}. 

Estimates \eqref{perm estimate L2}-\eqref{perm estimate L} for the loss operators follow. Estimates for the gain operators follow from \eqref{G leq L sharp} and the estimates for the loss operators. The proof is complete.
\end{proof}

Proposition \ref{permutation lemma} also implies an $L^1$-continuity result for the transported gain and loss operators
\begin{corollary}\label{L1 continuity} Let $0<T\leq\infty$ and $\alpha,\beta>0$. For $i\in\{1,2,3\}$, consider some sequences  $(f_{i,n})_{n}\subseteq\mathcal{F}_T^+$ and $f_i\in\mathcal{F}_T^+$ such that $f_{i,n}^\#(t)\overset{\mathcal{M}_{\alpha,\beta}}\longrightarrow f_i^\#(t)$ for all $t\in[0,T)$. Then, for all $t\in[0,T)$, the following convergence holds
\begin{equation}\label{convergence of L}
\left(L^\#(f_{1,n},f_{2,n},f_{3,n})(t),G^\#(f_{1,n},f_{2,n},f_{3,n})(t)\right)\overset{L^1_{x,v}} \longrightarrow \left(L^\#(f_1,f_2,f_3)(t),G^\#(f_1,f_2,f_3)(t)\right),\text{ as }n\to\infty.
\end{equation}
\end{corollary}
\begin{proof}  Fix $t\in [0,T)$. Since  $f^\#_{i,n}(t)\overset{\mathcal{M}_{\alpha,\beta}}\longrightarrow f^\#_i(t)$, for any $i\in\{1,2,3\}$, we have
\begin{equation}\label{convergence in sharp}
f_{i,n}^\#(t)\overset{a.e.}\longrightarrow f^\#_i(t),\quad \sup_{n\in\mathbb{N}}\{|f_{i,n}^\#(t)|,|f^\#_i(t)|\}\leq CM_{\alpha,\beta},
\end{equation}
for some constant $C>0$. Thus
\begin{equation}\label{convergence in non sharp}
f_{i,n}(t)\overset{a.e.}\longrightarrow f_i(t),\quad \sup_{n\in\mathbb{N}}\{|f_{i,n}(t)|,|f_i(t)|\}\leq CM_{\alpha,\beta}^{-\#}(t).
\end{equation}
Let us first prove \eqref{convergence of L} for the loss case. By \eqref{L} and triangle inequality, it suffices to prove
\begin{align}
\|L_2^\#(f_{1,n},f_{2,n})(t)-L_2^\#(f_1,f_2)(t)\|_{L^1_{x,v}}&\overset{n\to\infty} \longrightarrow 0,\label{L1 conv for L2}\\
\|L_3^\#(f_{1,n},f_{2,n},f_{3,n})(t)-L_3^\#(f_1,f_2,f_3)(t)\|_{L^1_{x,v}}&\overset{n\to\infty}\longrightarrow 0.\label{L1 conv for L3}
\end{align}
$\bullet$ Proof of \eqref{L1 conv for L2}:  Using bilinearity of $L_2^\#$, bound \eqref{convergence in non sharp} and monotonicity of $L_2^\#$, we have
\begin{align}
\|&L_2^\#(f_{1,n},f_{2,n})(t)-L_2^\#(f_1,f_2)(t)\|_{L^1_{x,v}}\leq\nonumber\\
&\leq \|L_2^\#(f_{1,n}-f_1,f_{2,n})(t)\|_{L^1_{x,v}}+\|L_2^\#(f_1,f_{2,n}-f_2)(t)\|_{L^1_{x,v}}\nonumber\\
&\leq C\|L_2^\#(f_{1,n}-f_1,M_{\alpha,\beta}^{-\#})(t)\|_{L^1_{x,v}}+C\|L_2^\#(M_{\alpha,\beta}^{-\#},f_{2,n}-f_2)(t)\|_{L^1_{x,v}}\nonumber\\
&\leq C_\beta(\|(f_{1,n}^\#(t)-f_1^\#(t))p_{\gamma_2,\gamma_3}\|_{L^1_{x,v}}+C_{\beta}\|(f_{2,n}^\#(t)-f_2^\#(t))p_{\gamma_2,\gamma_3}\|_{L^1_{x,v}}),\label{lemma perm for L2}
\end{align}
where  to obtain the last inequality we  use \eqref{perm estimate L2} from Proposition \ref{permutation lemma} and 
$\|M_{\alpha,\beta}\|_{\mathcal{M}_{\alpha,\beta}}=1.$

 By \eqref{convergence in sharp} and the Dominated Convergence Theorem, each of the terms in \eqref{lemma perm for L2} goes to zero as $n\to\infty$ and \eqref{L1 conv for L2} is proved.

$\bullet$ Proof of \eqref{L1 conv for L3}: Using trilinearity of $L_3^\#$, bound \eqref{convergence in non sharp} and monotonicity of $L_3^\#$, we have
\begin{align}
\|&L_3^\#(f_{1,n},f_{2,n},f_{3,n})(t)-L_3^\#(f_1,f_2,f_3)(t)\|_{L^1_{x,v}}\nonumber\\
&\leq \|L_3^\#(f_{1,n}-f_1,f_{2,n},f_{3,n})(t)\|_{L^1_{x,v}}+\|L_3^\#(f_1,f_{2,n}-f_2,f_{3,n})(t)\|_{L^1_{x,v}}+\|L_3^\#(f_1,f_2,f_{3,n}-f_3)\|_{L^1_{x,v}}\nonumber\\
&\leq C\|L_3^\#(f_{1,n}-f_1,M_{\alpha,\beta}^{-\#},M_{\alpha,\beta}^{-\#})(t)\|_{L^1_{x,v}}+C\|L_3^\#(M_{\alpha,\beta}^{-\#},f_{2,n}-f_2,M_{\alpha,\beta}^{-\#})(t)\|_{L^1_{x,v}}\nonumber\\
&\hspace{3cm}+C\|L_3^\#(M_{\alpha,\beta}^{-\#},M_{\alpha,\beta}^{-\#},f_{3,n}-f_3)\|_{L^1_{x,v}}\nonumber\\
&\leq C_\beta(\|(f_{1,n}^\#(t)-f_1^\#(t))p_{\gamma_2,\gamma_3}\|_{L^1_{x,v}}+C_\beta\|(f_{2,n}^\#(t)-f_2^\#(t))p_{\gamma_2,\gamma_3}\|_{L^1_{x,v}}\nonumber\\
&\hspace{3cm}+C_\beta\|(f_{3,n}^\#(t)-f_3^\#(t))p_{\gamma_2,\gamma_3}\|_{L^1_{x,v}}),\label{lemma perm for L3}
\end{align}
where to obtain the last inequality we use \eqref{perm estimate L3} from Proposition \ref{permutation lemma} and $\|M_{\alpha,\beta}\|_{\mathcal{M}_{\alpha,\beta}}=1.$ By \eqref{convergence in sharp} and the Dominated Convergence Theorem, each of the terms in \eqref{lemma perm for L3} goes to zero as $n\to\infty$ and \eqref{L1 conv for L3} is proved.  Combining \eqref{L1 conv for L2}-\eqref{L1 conv for L3}, we obtain \eqref{convergence of L}. 

The gain operator convergence follows with a similar argument.
\end{proof}

%

%

\subsection{A global estimate on the time average of the transported  gain and loss operators}\label{sub with transport} Here, we prove Proposition \ref{bounds on operators proposition}, which provides upper global bounds for the time average of the transported   operators. These estimates will be essential to prove that the necessary beginning condition \eqref{beginning condition}  for the convergence of the iteration  holds globally in time for small enough initial data (see Section \ref{sec: gwp}).  For the binary case and soft potentials, these bounds were established in \cite{alonso-gamba}. However, the presence of the ternary collisional operator requires new treatment which strongly relies on the properties of  ternary interactions.

Before stating Proposition \ref{bounds on operators proposition}, we provide the following auxiliary estimate for the time integral of a traveling Maxwellian which will be used in the proof of the result for $n=d$ in the binary case  and $n=2d$ in the ternary case.

 \begin{lemma}\label{time lemma} Let $n\in\mathbb{N}$, $x_0,u_0\in\mathbb{R}^n$, with $u_0\neq 0$ and $\alpha>0$. Then,  the following estimate holds
\begin{equation*}
\int_0^\infty e^{-\alpha|x_0-\tau u_0|^2}\,d\tau\leq \frac{\sqrt{\pi}}{2}\alpha^{-1/2} |u_0|^{-1}.
\end{equation*}
\end{lemma}
\begin{proof}
By triangle inequality, we have
$$\big|\tau\left|u_0\right|-\left|x_0\right|\big|\leq |x_0-\tau u_0|\Rightarrow e^{-\alpha|x_0-\tau u_0|^2}\leq e^{-\alpha(\tau| u_0|-|x_0|)^2},\quad\forall \tau\geq 0.$$
Therefore integrating in $\tau$, we obtain
\begin{align*}
\int_0^\infty e^{-\alpha|x_0-\tau u_0|^2}\,d\tau&\leq \int_0^\infty e^{-\alpha(\tau| u_0|-|x_0|)^2}\,d\tau\leq \alpha^{-1/2}|u_0|^{-1}\int_0^\infty e^{-y^2}\,dy\leq \frac{\sqrt{\pi}}{2}\alpha^{-1/2}|u_0|^{-1},
\end{align*}
and the estimate is proved.
\end{proof} 

We now state and prove Proposition \ref{bounds on operators proposition}. Given $f\in L^\infty([0,T),\mathcal{M}_{\alpha,\beta})$, recall from \eqref{time maxwellian norm} the norm
$$|||f|||_{\infty}=\sup_{t\in[0,T)}\|f(t)\|_{\mathcal{M}_{\alpha,\beta}}.$$

\begin{proposition}\label{bounds on operators proposition} 
Let  $0<T\leq\infty$ and $\alpha,\beta>0$.  Then, for all  $f,g,h\in\mathcal{F}_T$ with $f^\#,g^\#,h^\#\in L^\infty([0,T),\mathcal{M}_{\alpha,\beta})$, the following bounds hold for any $t\in[0,T)$

$\bullet$ For the binary operators

\begin{align}
\int_0^t |L_2^\#(f,g)(\tau)|\,d\tau,\text{ }\int_0^t |G_2^\#(f,g)(\tau)|\,d\tau&\leq K_\beta\alpha^{-1/2} M_{\alpha,\beta}||f^\#|||_{\infty}|||g^\#|||_{\infty}.\label{bound on binary loss}
\end{align}

$\bullet$ For the ternary operators
\begin{align}
\int_0^t |L_3^\#(f,g,h)(\tau)|\,d\tau,\text{ }\int_0^t |G_3^\#(f,g,h)(\tau)|\,d\tau&\leq K_\beta\alpha^{-1/2} M_{\alpha,\beta}|||f^\#|||_{\infty}|||g^\#|||_{\infty}|||h^\#|||_{\infty}.\label{bound on ternary loss}
\end{align}

$\bullet$ For the mixed operators

\begin{align}
\int_0^t|L^\#(f,g,h)|(\tau)\,d\tau,\text{ }\int_0^t|G^\#(f,g,h)(\tau)|\,d\tau&\leq K_\beta\alpha^{-1/2} M_{\alpha,\beta}|||f^\#|||_{\infty}|||g^\#|||_{\infty}(1+|||h^\#|||_{\infty}),\label{bound on total loss}
\end{align}

where
\begin{equation}
\begin{aligned}
K_\beta&=C_d\bigg[\|b_2\|_{L^1(\mathbb{S}^{d-1})}(\beta^{-d/2}+\frac{1}{d+\gamma_2-1})+\|b_3\|_{L^1(\mathbb{S}^{2d-1})}(\beta^{-d}+\frac{1}{2d+\gamma_3-1})\bigg].\label{constant K text}
\end{aligned}
\end{equation}
\end{proposition}

\begin{proof} 
We prove each of the estimates separately.

{\em Proof of \eqref{bound on binary loss}}: As mentioned above, these bounds were established for the soft potential case in \cite{alonso-gamba}. Here we also treat the hard potential case.  Since $L_2^\#, G_2^\#$ are bilinear, we may assume without loss of generality that 
\begin{equation}\label{norm condition binary} |||f^\#|||_{\infty}=|||g^\#|||_{\infty}=1.
\end{equation}
Let us first  prove it for the loss term. For any $t\in[0,T)$ and a.e. $(x,v)\in\mathbb{R}^{2d}$, relation \eqref{norm condition binary}, followed by an application of Lemma \ref{time lemma}  for $n=d$, $x_0=x$, $u_0=u$,  the fact that $-d+1<\gamma_2\leq 1$, and an application of part \textit{(i)} of Lemma \ref{convolution lemma} for $q_2=\gamma_2-1$ imply
\begin{align}
\int_0^t &|L_2^\#(f,g)(\tau,x,v)|\,d\tau\leq \|b_2\|_{L^1(\mathbb{S}^{d-1})}\int_0^t\int_{\mathbb{R}^d}|u|^{\gamma_2}|f(\tau,x+\tau v,v)||g(\tau,x+\tau v,v_1)|\,dv_1\,d\tau\nonumber\\
&=\|b_2\|_{L^1(\mathbb{S}^{d-1})}\int_0^t\int_{\mathbb{R}^d}|u|^{\gamma_2}|f^\#(\tau,x,v)||g^\#(\tau,x+\tau (v-v_1),v_1)|\,d\omega\,dv_1\,d\tau\nonumber\\
&\leq \|b_2\|_{L^1(\mathbb{S}^{d-1})} M_{\alpha,\beta}(x,v)\int_0^t\int_{\mathbb{R}^d}|u|^{\gamma_2} e^{-\alpha|x+\tau(v-v_1)|^2-\beta|v_1|^2}\,dv_1\,d\tau\nonumber\\
&\leq \|b_2\|_{L^1(\mathbb{S}^{d-1})} M_{\alpha,\beta}(x,v)\int_{\mathbb{R}^d}|u|^{\gamma_2}e^{-\beta |v_1|^2}\int_0^\infty e^{-\alpha|x-\tau u|^2}\,d\tau\,dv_1\nonumber\\
&\leq \|b_2\|_{L^1(\mathbb{S}^{d-1})}\frac{\sqrt{\pi}}{2}\alpha^{-1/2} M_{\alpha,\beta}(x,v)\int_{\mathbb{R}^d}|u|^{\gamma_2-1}e^{-\beta |v_1|^2}\,dv_1\nonumber\\
&\leq \|b_2\|_{L^1(\mathbb{S}^{d-1})}\frac{\sqrt{\pi}}{2}\widetilde{K}_{\beta,\gamma_2-1}^2\alpha^{-1/2} M_{\alpha,\beta}(x,v)\nonumber\\
&\leq C_d\|b_2\|_{L^1(\mathbb{S}^{d-1})}\alpha^{-1/2} \left(\beta^{-d/2}+\frac{1}{d+\gamma_2-1}\right)M_{\alpha,\beta}(x,v)\label{final estimate on binary loss},
\end{align}
where $C_d$ is an appropriate constant depending on the dimension $d$. To obtain \eqref{final estimate on binary loss}, we used \eqref{constant wide K_2} and the fact that $q_2=\gamma_2-1\leq 0$. Estimate \eqref{bound on binary loss}  for the loss term follows.

To prove \eqref{bound on binary loss} for the gain term, 
we will use the identity
\begin{equation}\label{identity binary}
|x+\tau(v-v')|^2+|x+\tau(v-v_1')|^2=|x|^2+|x+\tau(v-v_1)|^2,
\end{equation}
which follows from the binary conservation of momentum and energy
\begin{align}
v'+v_1'&=v+v_1,\nonumber\\
|v'|^2+|v_1'|^2&=|v|^2+|v_1|^2.\label{identity conservation binary}
\end{align}
For any $t\in[0,T)$ and a.e. $(x,v)\in\mathbb{R}^{2d}$, \eqref{norm condition binary} and \eqref{identity binary}-\eqref{identity conservation binary}  imply
\begin{align}
\int_0^t &|G_2^\#(f,g)(\tau,x,v)|\,d\tau
\leq \int_0^t\int_{\mathbb{S}^{d-1}\times\mathbb{R}^d}|u|^{\gamma_2}b_2(\hat{u}\cdot \omega)|f(\tau,x+\tau v,v')||g(\tau,x+\tau v,v_1')|\,d\omega\,dv_1\,d\tau\nonumber\\
&=\int_0^t\int_{\mathbb{S}^{d-1}\times\mathbb{R}^d}|u|^{\gamma_2}b_2(\hat{u}\cdot \omega)|f^\#(\tau,x+\tau(v-v'),v')||g^\#(\tau,x+\tau (v-v_1'),v_1')|\,d\omega\,dv_1\,d\tau\nonumber\\
&\leq\int_0^t\int_{\mathbb{S}^{d-1}\times\mathbb{R}^d}|u|^{\gamma_2}b_2(\hat{u}\cdot \omega)e^{-\alpha(|x+\tau(v-v')|^2+|x+\tau(v-v_1')|^2)}e^{-\beta(|v'|^2+|v_1'|^2)}\,d\omega\,dv_1\,d\tau\nonumber\\
&=\|b_2\|_{L^1(\mathbb{S}^{d-1})} M_{\alpha,\beta}(x,v)\int_0^t\int_{\mathbb{R}^d}|u|^{\gamma_2} e^{-\alpha|x+\tau(v-v_1)|^2-\beta|v_1|^2}\,dv_1\,d\tau\label{use of identities bin}.
\end{align}
  Combining \eqref{use of identities bin} with an identical argument to the one used for the loss term, we obtain
\begin{equation}\label{final estimate on binary gain}
\int_0^t |G_2^\#(f,g)(\tau,x,v)|\,d\tau\leq C_d\|b_2\|_{L^1(\mathbb{S}^{d-1})}\alpha^{-1/2} \left(\beta^{-d/2}+\frac{1}{d+\gamma_2-1}\right)M_{\alpha,\beta}(x,v),
\end{equation}
and estimate \eqref{bound on binary loss} for the gain term follows.

{\em Proof of \eqref{bound on ternary loss}}: Since $L_3^\#,G_3^\#$ is trilinear, we may assume without loss of generality that
\begin{equation}\label{norm condition ternary} |||f^\#|||_{\infty}=|||g^\#|||_{\infty}=|||h^\#|||_{\infty}=1.
\end{equation}
Let us first prove  \eqref{bound on ternary loss} for the loss term. For any $t\in[0,T)$ and a.e. $(x,v)\in\mathbb{R}^{2d}$,  \eqref{norm condition ternary} implies
\begin{align}
\int_0^t &|L_3^\#(f,g,h)(\tau,x,v)|\,d\tau\leq\nonumber\\
&\leq \|b_3\|_{L^1(\mathbb{S}^{2d-1})}\int_0^t\int_{\mathbb{R}^{2d}}|\bm{\widetilde{u}}|^{\gamma_3}|f(\tau,x+\tau v,v)||g(\tau,x+\tau v,v_1)||h(\tau,x+\tau v,v_2)|\,dv_1\,dv_2\,d\tau\nonumber\\
&=\|b_3\|_{L^1(\mathbb{S}^{2d-1})}\int_0^t\int_{\mathbb{R}^{2d}}|\bm{\widetilde{u}}|^{\gamma_3}|f^\#(\tau,x,v)||g^\#(\tau,x+\tau (v-v_1),v_1)||h^\#(\tau,x+\tau (v-v_2),v_2)|\,dv_1\,dv_2\,d\tau\nonumber\\
&\leq \|b_3\|_{L^1(\mathbb{S}^{2d-1})}M_{\alpha,\beta}(x,v)\int_0^t\int_{\mathbb{R}^{2d}}|\bm{\widetilde{u}}|^{\gamma_3}e^{-\alpha(|x+\tau(v-v_1)|^2+|x+\tau(v-v_2)|^2})e^{-\beta(|v_1|^2+|v_2|^2)}\,dv_1\,dv_2\,d\tau\nonumber\\
&\leq \|b_3\|_{L^1(\mathbb{S}^{2d-1})}  M_{\alpha,\beta}(x,v)\int_{\mathbb{R}^{2d}}|\bm{\widetilde{u}}|^{\gamma_3}e^{-\beta(|v_1|^2+|v_2|^2)}\int_0^\infty e^{-\alpha|\bm{x}-\tau\bm{u}|^2}\,d\tau\,dv_1\,dv_2,\label{vector notation ter loss}
\end{align}
 where in \eqref{vector notation ter loss} we use the notation
$$\bm{x}:=\begin{pmatrix}
x\\x
\end{pmatrix}\in\mathbb{R}^{2d},\quad\bm{u}=\begin{pmatrix}
v_1-v\\v_2-v
\end{pmatrix}\in\mathbb{R}^{2d}.$$
Notice that by triangle inequality and Young's inequality, we have
\begin{align}
|\bm{\widetilde{u}}|^2&=|v-v_1|^2+|v-v_2|^2+|v_1-v_2|^2\nonumber\\
&\leq |v-v_1|^2+|v-v_2|^2+(|v-v_1|+|v-v_2|)^2\nonumber\\
&\leq 3(|v-v_1|^2+|v-v_2|^2)\nonumber\\
&=3|\bm{u}|^2.\label{compared velocities}
\end{align}
Therefore, an application of Lemma \ref{time lemma} for $n=2d$, $x_0=\bm{x}$, $u_0=\bm{u}$, followed by  \eqref{compared velocities}, the fact that $-2d+1<\gamma_3\leq 1$, and an application of part \textit{(ii)} of Lemma \ref{convolution lemma} for $q_3=\gamma_3-1$  yield
\begin{align}
\int_0^t & |L_3^\#(f,g,h)(\tau,x,v)|\,d\tau\leq  \|b_3\|_{L^1(\mathbb{S}^{2d-1})} M_{\alpha,\beta}(x,v)\int_{\mathbb{R}^{2d}}|\bm{\widetilde{u}}|^{\gamma_3}e^{-\beta(|v_1|^2+|v_2|^2)}\int_0^\infty e^{-\alpha|\bm{x}-\tau\bm{u}|^2}\,d\tau\,dv_1\,dv_2\nonumber\\
&\leq   \|b_3\|_{L^1(\mathbb{S}^{2d-1})}\frac{\sqrt{\pi}}{2}\alpha^{-1/2}M_{\alpha,\beta}(x,v)\int_{\mathbb{R}^{2d}}|\bm{\widetilde{u}}|^{\gamma_3}|\bm{u}|^{-1}e^{-\beta(|v_1|^2+|v_2|^2)}\,dv_1\,dv_2\nonumber\\
&\leq \|b_3\|_{L^1(\mathbb{S}^{2d-1})}\frac{\sqrt{3\pi}}{6}\alpha^{-1/2}M_{\alpha,\beta}(x,v)\int_{\mathbb{R}^{2d}}|\bm{\widetilde{u}|^{\gamma_3-1}}e^{-\beta(|v_1|^2+|v_2|^2)}\,dv_1\,dv_2\nonumber\\
&\leq \|b_3\|_{L^1(\mathbb{S}^{2d-1})}\frac{\sqrt{3\pi}}{6}\widetilde{K}_{\beta,\gamma_3-1}^3\alpha^{-1/2}M_{\alpha,\beta}(x,v)\nonumber\\
&\leq C_d  \|b_3\|_{L^1(\mathbb{S}^{2d-1})}\alpha^{-1/2} \left(\beta^{-d}+\frac{1}{2d+\gamma_3-1}\right)M_{\alpha,\beta}(x,v)\label{final estimate on ternary loss},
\end{align}
where $C_d$ is an appropriate constant depending on the dimension $d$. To obtain \eqref{final estimate on ternary loss}, we used \eqref{constant wide K_3} and the fact that $q_3=\gamma_3-1\leq 0$. Estimate \eqref{bound on ternary loss}  for the loss term follows.

To prove \eqref{bound on ternary loss} for the gain term,
we will use the identity
\begin{equation}\label{identity ter}
|x+\tau(v-v^*)|^2+|x+\tau(v-v_1^*)|^2+|x+\tau(v-v_2^*)|^2=|x|^2+|x+\tau(v-v_1)|^2+|x+\tau(v-v_2)|^2,
\end{equation}
following from the ternary conservation of momentum and energy
\begin{align}
v^*+v_1^*+v_2^*&=v+v_1+v_2,\nonumber\\
|v^*|^2+|v_1^*|^2+|v_2^*|^2&=|v|^2+|v_1|^2+|v_2|^2.\label{identity cons ter}
\end{align}
For any $t\in[0,T)$ and a.e. $(x,v)\in\mathbb{R}^{2d}$, by \eqref{norm condition ternary} and \eqref{identity ter}-\eqref{identity cons ter}, we obtain
\begin{align}
\int_0^t &|G_3^\#(f,g,h)(\tau)|\,d\tau\nonumber\\
&\leq \int_0^t\int_{\mathbb{S}^{2d-1}\times\mathbb{R}^{2d}}|\bm{\widetilde{u}}|^{\gamma_3}b_3(\bm{\bar{u}\cdot\omega},\omega_1\cdot\omega_2)| f(\tau,x+\tau v,v^*)||g(\tau,x+\tau v,v_1^*)||h(\tau,x+\tau v,v_2^*)|\nonumber\\
&\hspace{8cm}\times\,d\omega_1\,d\omega_2\,dv_1\,dv_2\,d\tau\nonumber\\
&=\int_0^t\int_{\mathbb{S}^{2d-1}\times\mathbb{R}^{2d}}|\bm{\widetilde{u}}|^{\gamma_3}b_3(\bm{\bar{u}\cdot\omega},\omega_1\cdot\omega_2)|f^\#(\tau,x+\tau (v-v^*),v^*)||g^\#(\tau,x+\tau (v-v_1^*),v_1^*)|\nonumber\\
&\hspace{6cm}\times|h^\#(\tau,x+\tau (v-v_2^*),v_2^*)|\,d\omega_1\,d\omega_2\,dv_1\,dv_2\,d\tau\nonumber\\
&\leq \int_0^t\int_{\mathbb{S}^{2d-1}\times\mathbb{R}^{2d}}|\bm{\widetilde{u}}|^{\gamma_3}b_3(\bm{\bar{u}\cdot\omega},\omega_1\cdot\omega_2)e^{-\alpha(|x+\tau(v-v^*)|^2+|x+\tau(v-v_1^*)|^2+|x+\tau(v-v_2^*)|^2)}e^{-\beta(|v^*|^2+|v_1^*|^2+|v_2^*|^2)}\nonumber\\
&\hspace{8cm}\times\,d\omega_1\,d\omega_2\,dv_1\,dv_2\,d\tau\nonumber\\
&=\|b_3\|_{L^1(\mathbb{S}^{2d-1})}M_{\alpha,\beta}(x,v)\int_0^t\int_{\mathbb{R}^{2d}}|\bm{\widetilde{u}}|^{\gamma_3}e^{-\alpha(|x+\tau(v-v_1)|^2+|x+\tau(v-v_2)|^2)}e^{-\beta(|v_1|^2+|v_2|^2)}\,dv_1\,dv_2\,d\tau\nonumber\\
&\leq \|b_3\|_{L^1(\mathbb{S}^{2d-1})}  M_{\alpha,\beta}(x,v)\int_{\mathbb{R}^{2d}}|\bm{\widetilde{u}}|^{\gamma_3}e^{-\beta(|v_1|^2+|v_2|^2)}\int_0^\infty e^{-\alpha|\bm{x}-\tau\bm{u}|^2}\,d\tau\,dv_1\,dv_2.\label{vector notation ter}
\end{align}
  Combining \eqref{vector notation ter}  with an  identical argument to the one used for the loss case, we obtain
\begin{equation}\label{final bound for ternary gain}
\int_0^t |G_3^\#(f,g,h)(\tau)|\,d\tau\leq C_d  \|b_3\|_{L^1(\mathbb{S}^{2d-1})}\alpha^{-1/2} \left(\beta^{-d}+\frac{1}{2d+\gamma_3-1}\right)M_{\alpha,\beta}(x,v),
\end{equation}
and estimate \eqref{bound on ternary loss} for the gain term follows.

{\em Proof of \eqref{bound on total loss}}:  It follows directly from \eqref{bound on binary loss}-\eqref{bound on ternary loss}.
\end{proof}

\section{The Kaniel-Shinbrot iteration scheme and the associated linear problem}\label{sec iteration}
In this section, we present the Kaniel-Shinbrot iteration scheme which will then be used as the heart of the construction of  a global solution in Section \ref{sec: gwp}. This scheme is motivated by the works of \cite{illner-shinbrot,kaniel-shinbrot}. However the presence of the ternary collisional operator, in addition to the binary collisional operator, required a modification of the original construction. 

In particular, we outline the construction of the  Kaniel-Shinbrot iteration that we will use in this paper.  Formally speaking, given an initial data $f_0$, we construct an increasing sequence $(l_n)_{n\in\mathbb{N}}$ and a decreasing sequence $(u_n)_{n\in\mathbb{N}}$, with $l_n\leq u_n$, 
through the iteration
\begin{equation}\label{diffeq for loss n}
\begin{aligned}
\frac{\,d l_n}{\,dt}+ v\cdot \nabla_x l_n &=G(l_{n-1},l_{n-1},l_{n-1}) - L(l_n,u_{n-1},u_{n-1}),\\
l_n(0)&=f_0,
\end{aligned}
\end{equation}
\begin{equation}\label{diffeq for gain n}
\begin{aligned}
\frac{\,d u_n}{\,dt} + v\cdot \nabla_x u_n &=G(u_{n-1},u_{n-1},u_{n-1}) -  L(u_n,l_{n-1},l_{n-1}),\\
u_{n}(0)&=f_0.
\end{aligned}
\end{equation}
We will see that that the sequences $l_n$, $u_n$ converge to the same limit, namely a function $f$, which will be the solution of the binary-ternary Boltzmann equation \eqref{generalized boltzmann equation}. 

To make things rigorous, we first study an associated linear problem, and then inductively  apply these results, together with the  estimates derived in Section \ref{sec: properties of gain and loss}, to  establish that the Kaniel-Shinbrot iteration converges to a solution of \eqref{generalized boltzmann equation}, provided that an appropriate beginning condition is satisfied. This solution will be unique in the class of functions uniformly bounded by a Maxwellian.

\subsection{The associated linear problem }\label{sec aux}
Here, we  prove well-posedness for a linear problem associated to the iteration scheme \eqref{diffeq for loss n}-\eqref{diffeq for gain n}.
More precisely, given some functions of time $g,h$,  we show well-posedness up to time $0<T\leq\infty$ of the  linear problem
\begin{equation}\label{linear problem no sharps}
\begin{cases}
\partial_tf+v\cdot\nabla_xf=h-L(f,g,g),\quad (t,x,v)\in (0,T)\times\mathbb{R}^d\times\mathbb{R}^d,\\
f(0)=f_0,\quad (x,v)\in\mathbb{R}^d\times\mathbb{R}^d.
\end{cases}
\end{equation}
\begin{definition}\label{def of linear solution} Let $0<T\leq\infty$, $\alpha,\beta>0$, $f_0\in L^{1,+}_{x,v}$,  $g^\#\in L^\infty([0,T),\mathcal{M}_{\alpha,\beta}^+)$ and $h^\#\in L^1_{loc}([0,T),L^{1,+}_{x,v})$. We say that a function  $f\in \mathcal{F}_T^+$ with
\begin{enumerate}[(i)]
\item $f^\#\in C^0([0,T),L^{1,+}_{x,v}),$
\item $L^\#(f,g,g)\in L^1_{loc}([0,T),L^{1,+}_{x,v}),$ 

\item $f^\#$ is weakly differentiable and satisfies 
\begin{equation}\label{linear problem}
\begin{cases}
\displaystyle\frac{\,df^\#}{\,dt}+L^\#(f,g,g)=h^\#,\\
f^\#(0)=f_0,
\end{cases}
\end{equation}
\end{enumerate}
is a mild solution of \eqref{linear problem no sharps} in $[0,T)$ with initial data $f_0\in L^{1,+}_{x,v}$.
\end{definition}
\begin{remark} The differential equation of \eqref{linear problem} is interpreted as an equality of distributions since all terms involved belong to $L^1_{loc}([0,T),L^{1,+}_{x,v})$.
\end{remark}
\begin{remark} Remarks \ref{remark on measure preserving}-\ref{isometry remark}  imply that a mild solution $f$ to \eqref{linear problem no sharps} belongs to $C^0([0,T),L^{1,+}_{x,v})$.
\end{remark}
For technical reasons, we first prove well-posedness of \eqref{linear problem no sharps} under the additional assumptions
\begin{equation}\label{additional assumption on h}
f_0\in\mathcal{M}_{\alpha,\beta}^+,\quad
0\leq  h^\#(t)\leq Ce^{-t^2}M_{\alpha,\beta},\quad\forall t\in[0,T),
\end{equation}
for some constant $C>0$. Clearly if   \eqref{additional assumption on h} holds, then $f_0\in L^{1,+}_{x,v}$ and $h^\#\in L^1_{loc}([0,T),L^{1,+}_{x,v})$, thus \eqref{additional assumption on h} is  a stronger assumption  than those appearing in Definition \ref{def of linear solution}. This additional assumption will be removed later using an approximation argument.
\begin{lemma}\label{linear lemma} Let  $0<T\leq\infty$ and $\alpha,\beta>0$. Consider $f_0,h$ satisfying \eqref{additional assumption on h} and $g^\#\in L^\infty([0,T),\mathcal{M}_{\alpha,\beta}^+)$. Then, there exists a  mild solution $f$ of \eqref{linear problem no sharps} with $f^\#\in L^\infty([0,T),\mathcal{M}_{\alpha,\beta}^+)$.
Moreover,  $\|f^\#(\cdot)\|_{L^1_{x,v}}$ is absolutely continuous and satisfies
\begin{equation}\label{equation on linear derivative lemma}
\|f^\#(t)\|_{L^1_{x,v}}+\int_0^t\|L^\#(f,g,g)(\tau)\|_{L^1_{x,v}}\,d\tau=\|f_0\|_{L^1_{x,v}}+\int_0^t\|h^\#(\tau)\|_{L^1_{x,v}}\,d\tau,\quad\forall t\in[0,T).
\end{equation}
\end{lemma}
\begin{proof} Since $g^\#\in L^\infty([0,T),\mathcal{M}_{\alpha,\beta}^+)$,  part {\em (i)} of Lemma \ref{bound on R tilde lemma} implies
\begin{equation}\label{guarantee condition R}
0\leq R^\#(g,g)(t)\leq C_\beta |||g^\#|||_{\infty}(1+|||g^\#|||_\infty)p_{\gamma_2,\gamma_3},\quad\forall t\in [0,T),
\end{equation}
for some constant $C_\beta>0$ depending on $\beta$.
We define $f$ by
\begin{equation}\label{f  linear}
f^\#(t):=
f_0 \exp\left(-\displaystyle\int_0^t R^\#(g,g)(\sigma)\,d\sigma\right)+\displaystyle\int_0^th^\#(\tau)\exp\left(-\displaystyle\int_\tau^tR^\#(g,g)(\sigma)\,d\sigma\right)\,d\tau,\quad t\in[0,T).
\end{equation}
By \eqref{additional assumption on h}, \eqref{guarantee condition R}, and the fact  $f_0\in\mathcal{M}_{\alpha,\beta}^+$, $f^\#$ is well-defined and satisfies the bound
\begin{equation}\label{elemntary bound on f}
0\leq f^\#(t)\leq f_0+\int_0^th^\#(\tau)\,d\tau\leq \left(\|f_0\|_{\mathcal{M}_{\alpha,\beta}}+C\frac{\sqrt{\pi}}{2}\right)M_{\alpha,\beta},\quad\forall t\in[0,T),
\end{equation}
thus $f\geq 0$ and
\begin{equation}\label{f tilde in L inf}
f^\#\in L^\infty([0,T),\mathcal{M}_{\alpha,\beta}^+).
\end{equation}
Let  us now show that $f^\#\in C^0([0,T),L^{1,+}_{x,v})$. For any $t,s\in[0,T)$, expression \eqref{f  linear} yields
\begin{align*}
|f^\#(t)-f^\#(s)|&=\bigg|\left[f_0\exp\left(-\int_0^sR^\#(g,g)(\sigma)\,d\sigma\right)+\int_0^s h^\#(\tau)\exp\left(-\int_\tau^sR^\#(g,g)(\sigma)\,d\sigma\right)\,d\tau\right]\\
&\hspace{1cm}\times\left[\exp\left(-\int_s^tR^\#(g,g)(\sigma)\,d\sigma\right)-1\right]+\int_s^t h^\#(\tau)\exp\left(-\int_\tau^tR^\#(g,g)(\sigma)\,d\sigma\right)\,d\tau\bigg|,
\end{align*}
therefore by \eqref{additional assumption on h}, \eqref{guarantee condition R}, we may find a positive constants $C_{f_0,g,h}>0$  such that
\begin{equation}
|f^\#(t)-f^\#(s)|\leq C_{f_0,g,h} M_{\alpha,\beta}(1-e^{-C_{f_0,g,h}|t-s|p_{\gamma_2,\gamma_3}})+C_{f_0,g,h}|t-s|M_{\alpha,\beta},\quad\forall t\in[0,T).
\end{equation}
Using the elementary inequality
$1-e^{-x}\leq x,$ for all $x\geq 0$,
we obtain
\begin{equation}\label{before integration}
|f^\#(t)-f^\#(s)|\leq 2C_{f_0,g,h}|t-s|p_{\gamma_2,\gamma_3}M_{\alpha,\beta},\quad\forall t\in[0,T).
\end{equation}
Integrating \eqref{before integration},  we obtain
\begin{equation}\label{lipshitz}
 \|f^\#(t)-f^\#(s)\|_{L^1_{x,v}}\leq 2C_{f_0,g,h}|t-s|,\quad\forall t,s\in [0,T),
\end{equation}
since $p_{\gamma_2,\gamma_3}M_{\alpha,\beta}\in L^{1,+}_{x,v}$. We conclude that $f^\#\in C^0([0,T), L^{1,+}_{x,v})$, therefore $f\in C^0([0,T), L^{1,+}_{x,v})$. In particular, bound \eqref{lipshitz} implies that  $f$ is actually Lipschitz continuous.

Since  $f^\#,g^\#\in L^\infty([0,T),\mathcal{M}_{\alpha,\beta}^+)$,  part \textit{(ii)} of Lemma  \ref{bound on R tilde lemma}  implies
\begin{equation}\label{L in L_inf}
L^\#(f,g,g)\in L^\infty([0,T),L^{1,+}_{x,v})\subseteq L^1_{loc}([0,T),L^{1,+}_{x,v}) .
\end{equation}

Finally, by \eqref{additional assumption on h}, \eqref{L in L_inf}, representation \eqref{f  linear} and the Dominated Convergence Theorem, we conclude that $f^\#$ is weakly differentiable and satisfies
\begin{equation}\label{diffeq linear}
\begin{cases}
\displaystyle\frac{\,df^\#}{\,dt}+L^\#(f,g,g)=h^\#,\\
f^\#(0)=f_0,
\end{cases}
\end{equation}
 thus it is a mild solution of \eqref{linear problem no sharps}.

Integrating \eqref{diffeq linear}, the Fundamental Theorem of Calculus and the fact that $f^\#\in C^0([0,T),L^{1,+}_{x,v})$, $L^\#(f,g,g)$ and $h^\#\in L^1_{loc}([0,T),L^{1,+}_{x,v})$, imply
\begin{equation}\label{integral equation linear pre}
f^\#(t)+\int_0^tL^\#(f,g,g)(\tau)\,d\tau=f_0+\int_0^th^\#(\tau)\,d\tau,\quad\forall t\in[0,T).
\end{equation}
Using non-negativity of all terms involved in \eqref{integral equation linear pre} and Fubini's Theorem, we obtain \eqref{equation on linear derivative lemma} and absolute continuity of $\|f(t)\|_{L^1_{x,v}}$ follows. The proof is complete.
\end{proof}

Since the gain operator does not satisfy \eqref{additional assumption on h}, it will be convenient to relax  assumption \eqref{additional assumption on h} to $f_0\in L^{1,+}_{x,v}$, $h^\#\in L^1_{loc}([0,T),L^{1,+}_{x,v})$. As in \cite{kaniel-shinbrot}, the idea is to approximate  $f_0, h^\#$ in the $L^1_{x,v}$-norm with a monotone sequence of solutions occurring from a repeated application of Lemma \ref{linear lemma}.
We obtain the following well-posedness result
\begin{proposition}\label{linear prop}
 Let  $0<T\leq\infty$ and $\alpha,\beta>0$. Consider $f_0\in L^{1,+}_{x,v}$, $g^\#\in L^\infty([0,T),\mathcal{M}_{\alpha,\beta}^+)$ and $h^\#\in L^1_{loc}([0,T),L^{1,+}_{x,v})$. Then, there exists a unique  mild solution $f$ of \eqref{linear problem no sharps}. In particular  $f^\#$ is given by
 \begin{equation}\label{explicit formula for f}
 f^\#(t):=
f_0 \exp\left(-\displaystyle\int_0^t R^\#(g,g)(\sigma)\,d\sigma\right)+\displaystyle\int_0^th^\#(\tau)\exp\left(-\displaystyle\int_\tau^tR^\#(g,g)(\sigma)\,d\sigma\right)\,d\tau,\quad t\in[0,T).
 \end{equation}
\end{proposition}
\begin{proof} \textit{Existence}: Given $n\in\mathbb{N}$, let us define 
\begin{equation}\label{chopped function phi}
 f_{0,n}:=
\begin{cases}
f_0,\text{ if } f_0\leq nM_{\alpha,\beta},\\
nM_{\alpha,\beta},\text{ if }f_0>nM_{\alpha,\beta},
\end{cases}
\end{equation}
and
\begin{equation}\label{chopped function}
 h_n^\#(t):=
\begin{cases}
h^\#(t),\text{ if }h^\#(t)\leq ne^{-t^2}M_{\alpha,\beta},\\
ne^{-t^2}M_{\alpha,\beta},\text{ if }h^\#(t)>ne^{-t^2}M_{\alpha,\beta}. 
\end{cases}
\end{equation}
It is clear that $f_{0,n}, h_n$ satisfy condition \eqref{additional assumption on h} for all $n\in\mathbb{N}$ and that
\begin{align}
0&\leq f_{0,n}\nearrow f_0\quad\text{as }n\to\infty,\label{monotone approximation cut phi}\\
\forall t\in[0,T):\quad 0&\leq h_n^\#(t)\nearrow h^\#(t)\quad\text{as }n\to\infty.\label{monotone approximation cut}
\end{align}
  Then the Monotone Convergence Theorem yields that 
  \begin{align}  
 &\|f_{0,n}\|_{L^1_{x,v}}\nearrow \|f_0\|_{L^1_{x,v}},\quad\text{as }n\to\infty,\label{phi_n conv of norms}\\
\forall t\in [0,T):\quad & \|h_n^\#(t)\|_{L^1_{x,v}}\nearrow \|h^\#(t)\|_{L^1_{x,v}},\quad\text{as }n\to\infty. \label{h_n conv of norms}
  \end{align}
  Moreover, since $f_0\in L^1_{x,v}$ and $h^\#\in L^1_{loc}([0,T),L^{1,+}_{x,v})$, relations \eqref{monotone approximation cut phi}-\eqref{monotone approximation cut} and the Dominated Convergence Theorem yield
  \begin{align} 
  f_{0,n}\overset{L^1_{x,v}}\longrightarrow &\text{ } f_0,\quad\text{as }n\to\infty,\label{phi_n to phi}\\
  \text{for a.e. }t\in[0,T):\quad h_n^\#(t)\overset{L^1_{x,v}}\longrightarrow &\text{ } h^\#(t),\quad\text{as }n\to\infty,\label{h_n to h}\\
  \forall t\in[0,T):\quad \int_0^t h_n^\#(\tau)\,d\tau\overset{L^1_{x,v}}\longrightarrow & \int_0^t h^\#(\tau)\,d\tau,\quad\text{as }n\to\infty.\label{h_n to h in time}
  \end{align}
  Let $f_n\in\mathcal{F}^+_T$  be the  mild solution to the problem
\begin{equation}\label{linear problem n no sharps}
\begin{cases}
\displaystyle\frac{\,df_n}{\,dt} + v\cdot \nabla_x f_n= h_n - L(f_n,g,g),\\
f_n(0)=f_{0,n},
\end{cases}
\end{equation}
constructed in Lemma \ref{linear lemma}.  Let us note that Lemma \ref{linear lemma} is applicable  for all $n\in\mathbb{N}$   since $f_{0,n}, h_n$ satisfy \eqref{additional assumption on h}. Hence, $f_n^\#$ satisfies
\begin{equation}\label{linear problem n}
\begin{cases}
\displaystyle\frac{\,df_n^\#}{\,dt}+L^\#(f_n,g,g)= h_n^\#,\\
f_n^\#(0)=f_{0,n},
\end{cases}
\end{equation}
and is given by the formula
 \begin{equation}\label{explicit formula for f_n}
 f_n^\#(t):=
f_{0,n} \exp\left(-\displaystyle\int_0^t R^\#(g,g)(\sigma)\,d\sigma\right)+\displaystyle\int_0^th_n^\#(\tau)\exp\left(-\displaystyle\int_\tau^tR^\#(g,g)(\sigma)\,d\sigma\right)\,d\tau,\quad t\in[0,T).
 \end{equation}
  Also by \eqref{equation on linear derivative lemma}, given $t\in[0,T)$, we have the bound
\begin{align}
\sup_{n\in\mathbb{N}}\|f^\#_n(t)\|_{L^1_{x,v}}&\leq \sup_{n\in\mathbb{N}}\left( \|f_{0,n}\|_{L^1_{x,v}}+\int_0^t\|h_n^\#(\tau)\|_{L^1_{x,v}}\,d\tau\right)\leq \|f_0\|_{L^1_{x,v}}+\int_0^t\|h^\#(\tau)\|_{L^1_{x,v}}\,d\tau<\infty,\label{less than inf}
\end{align}
where to obtain the last bound we use \eqref{phi_n conv of norms}-\eqref{h_n conv of norms}, the fact that $R^\#(g,g)\geq 0$ (by monotonicity of $R^\#$ and $g\geq 0$), $f_0\in L^{1}_{x,v}$ and $h^\#\in L^1_{loc}([0,T),L^{1,+}_{x,v})$.

Since the sequences  $(f_{0,n})_{n}$, $(h_n^\#(t))_{n}$ are increasing and non-negative for all $t\in[0,T)$, formula \eqref{explicit formula for f_n} implies that the sequence $(f_n^\#(t))_{n}$ is increasing for all $t\in[0,T)$. Let us define 
$$f^\#(t):=\lim_{n\to\infty} f_n^\#(t).$$ 
Clearly $f\geq 0$.
By the Monotone Convergence Theorem and bound \eqref{less than inf} we obtain  that $f^\#(t)\in L^{1,+}_{x,v},\quad\forall t\in [0,T)$.
 Then, the Dominated Convergence Theorem implies
\begin{equation}\label{L1 convergence f_n to f}
\forall t\in [0,T):\quad f_n^\#(t)\overset{L^1_{x,v}}\longrightarrow f^\#(t),\quad\text{as }n\to\infty.
\end{equation}
Moreover,  we have
\begin{equation}\label{monotone convergence Ln}
\forall t\in[0,T):\quad L^\#(f_n,g,g)(t)=f_n^\#(t)R^\#(g,g)(t)\nearrow f^\#(t)R^\#(g,g)(t)=L^\#(f,g,g)(t),\quad\text{as }n\to\infty,
\end{equation}
since $R^\#(g,g)(t)\geq 0$ by monotonicity of $R^\#$ and the fact that $g\geq 0$. By the Monotone Convergence Theorem, we obtain
\begin{equation}\label{monotone convergence Ln L1}
\forall t\in[0,T):\quad \int_0^t\|L^\#(f_n,g,g)(\tau)\|_{L^1_{x,v}}\,d\tau\nearrow\int_0^t\|L^\#(f,g,g)(\tau)\|_{L^1_{x,v}}\,d\tau,\quad\text{as }n\to\infty.
\end{equation}
Therefore, for any $t\in[0,T)$, equation \eqref{equation on linear derivative lemma}, implies
\begin{align}
\int_0^t\|L^\#(f,g,g)(\tau)\|_{L^1_{x,v}}\,d\tau&=\sup_{n\in\mathbb{N}}\int_0^t\|L^\#(f_n,g,g)(\tau)\|_{L^1_{x,v}}\,d\tau\\
&\leq \sup_{n\in\mathbb{N}}\left(\|f_{0,n}\|_{L^1_{x,v}}+\int_0^t\|h_n^\#(\tau)\|_{L^1_{x,v}}\,d\tau\right)\nonumber\\
&\leq \|f_0\|_{L^1_{x,v}}+\int_0^t\|h^\#(\tau)\|_{L^1_{x,v}}\,d\tau<\infty,\label{less than inf 2}
\end{align}
since $f_0\in L^{1}_{x,v}$ and $h^\#\in L^1_{loc}([0,T),L^{1,+}_{x,v})$,
thus
\begin{align}
L^\#(f,g,g)(t)&\in L^1_{x,v},\quad\text{for a.e. }t\in [0,T),\label{in L1 a.e.}\\
L^\#(f,g,g)&\in L^1_{loc}([0,T),L^{1,+}_{x,v}).\label{L in L1_loc}
\end{align}
By \eqref{monotone convergence Ln}, \eqref{in L1 a.e.} and the Dominated Convergence Theorem, for a.e. $t\in[0,T)$, we have
\begin{equation}\label{L_n to L in L1}
L^\#(f_n,g,g)(t)\overset{L^1_{x,v}}\longrightarrow L^\#(f,g,g)(t),\quad\text{as $n\to\infty$},
\end{equation}
and by \eqref{L in L1_loc} and another application of the Dominated Convergence Theorem, we obtain
\begin{equation}\label{conv of L integrals}
\int_0^t L^\#(f_n,g,g)(\tau)\,d\tau\overset{L^1_{x,v}}\longrightarrow \int_0^tL^\#(f,g,g)(\tau)\,d\tau,\quad\forall t\in[0,T).
\end{equation}
Since $f_n^\#$ satisfies \eqref{linear problem n},  the Fundamental Theorem of Calculus and the fact that $f_n^\#\in C^0([0,T),L^{1,+}_{x,v})$, $L^\#(f_n,g,g)$ and $h_n^\#\in L^1_{loc}([0,T),L^{1,+}_{x,v})$ imply
\begin{equation}\label{integral equation n}
f_n^\#(t)+\int_0^tL^\#(f_n,g,g)(\tau)\,d\tau=f_{0,n}+\int_0^t h_n^\#(\tau)\,d\tau,\quad\forall t\in[0,T),\quad\forall n\in\mathbb{N}.
\end{equation}
Using \eqref{L1 convergence f_n to f}, \eqref{conv of L integrals}, \eqref{phi_n to phi}, and \eqref{h_n to h in time}, we let $n\to\infty$ in \eqref{integral equation n} to obtain
\begin{equation}\label{integral equation final}
f^\#(t)+\int_0^tL^\#(f,g,g)(\tau)\,d\tau=f_0+\int_0^t h^\#(\tau)\,d\tau,\quad\forall t\in[0,T),\quad\forall n\in\mathbb{N},
\end{equation}
thus $f^\#\in C^0([0,T),L^{1,+}_{x,v})$, $f^\#$ is weakly differentiable and  satisfies \eqref{linear problem}. We conclude that $f$ is a mild solution of \eqref{linear problem no sharps}. Moreover, since $g\geq 0$,  we may take the limit as $n\to\infty$ in both sides of \eqref{explicit formula for f_n}  to obtain \eqref{explicit formula for f}.

 {\em Uniqueness}: Since the problem is linear it suffices to show that if $f$ is a solution of \eqref{linear problem no sharps} with $f_0=0$ and $h=0$, then $f=0$.

Assume $f$ is a mild solution of \eqref{linear problem no sharps}  with $f_0=0$ and $h=0$ i.e. $f\geq 0$, $f^\#\in C^0([0,T),L^{1,+}_{x,v})$, $L^\#(f,g,g)\in L^1_{loc}([0,T),L^{1,+}_{x,v})$ and $f^\#$ is weakly differentiable and satisfies 
\begin{equation}\label{linear problem uniqueness}
\begin{cases}
\displaystyle\frac{\,d f^{\#}}{\,dt}+L^\#(f,g,g)=0,\\
f^\#(0)=0.
\end{cases}
\end{equation}
Then  \eqref{linear problem uniqueness}, the Fundamental Theorem of Calculus and the facts $f^\#\in C^0([0,T),L^{1,+}_{x,v})$, $L^\#(f,g,g)\in L^1_{loc}([0,T),L^{1,+}_{x,v})$ imply
\begin{equation}\label{integral f linear}
f^\#(t)=-\int_0^t L^\#(f,g,g)(\tau)\,d\tau=-\int_0^t f^\#(\tau)R^\#(g,g)(\tau)\,d\tau,\quad\forall t\in[0,T).
\end{equation}
We claim the following

\textbf{Claim}: For any compact set $K\subseteq\mathbb{R}^d\times\mathbb{R}^d$, we have $\|f^\#(t)\|_{L^1_{x,v}(K)}=0,\quad\forall t\in[0,T)$.
 
\textbf{Proof of the claim}: Fix any compact set $K\subseteq\mathbb{R}^d\times\mathbb{R}^d$. By \eqref{integral f linear}, Fubini's Theorem, part \textit{(i)} of Lemma \ref{bound on R tilde lemma} and the fact that $p_{\gamma_2,\gamma_3}$ is continuous,  we obtain
\begin{align}
\|f^\#(t)\|_{L^1_{x,v}(K)}&\leq\int_0^t \|f^\#(\tau)R^\#(g,g)(\tau)\|_{L^1_{x,v}(K)}\,d\tau\nonumber\\
&\leq C_\beta|||g^\#|||_\infty(1+|||g^\#|||_\infty)\int_0^t \|p_{\gamma_2,\gamma_3}f^\#(\tau)\|_{L^1_{x,v}(K)}\,d\tau\nonumber\\
&\leq C_{K,\beta}|||g^\#|||_\infty(1+|||g^\#|||_\infty)\int_0^t\|f^\#(\tau)\|_{L^1_{x,v}(K)}\,d\tau.\label{continuity bound Gronwall}
\end{align}
 Since $f^\#\in C^0([0,T),L^{1,+}_{x,v})$,  the map $t\in[0,T)\to \|f^\#(t)\|_{L^1_{x,v}(K)}\in [0,\infty)$ is continuous, thus \eqref{continuity bound Gronwall} and Gronwall's inequality imply that 
 $$\|f^\#(t)\|_{L^1_{x,v}(K)}=0,\quad\forall t\in [0,T).$$
 The claim is proved.
 
 Consider now a sequence of compact sets $(K_m)_{m}\nearrow\mathbb{R}^d\times\mathbb{R}^d$. By the claim above, and the Monotone Convergence Theorem, we have
 $$\|f^\#(t)\|_{L^1_{x,v}}=\lim_{m\to\infty} \|f^\#(t)\|_{L^1_{x,v}(K_m)}=0,\quad\forall t\in [0,T).$$
 Since $f^\#\geq 0$, we obtain $f^\#=0$ and hence $f=0$.
 Uniqueness is proved.
\end{proof}

The following comparison Corollary comes immediately by the monotonicity of $R^\#$ and    representation \eqref{explicit formula for f}.
\begin{corollary}\label{comparision corollary from prop} Let  $0<T\leq\infty$ and $\alpha,\beta>0$. Consider $f_{0,1},f_{0,2}\in L^{1,+}_{x,v}$, $g_1,g_2\in L^{\infty}([0,T),\mathcal{M}_{\alpha,\beta}^+)$ and $h_1,h_2\in L_{loc}^{1}([0,T),L^{1,+}_{x,v})$ with
$$ f_{0,1}\leq f_{0,2},\quad g_1^\#\geq g_2^\#,\quad h_1^\#\leq h_2^\#.$$
Let $f_i$, $i\in\{1,2\}$ be the corresponding unique solution of \eqref{linear problem no sharps} with $f_0:=f_{0,i}$, $g:=g_i$ and $h:=h_i$. Then
$f_1\leq f_2.$
\end{corollary}
\begin{proof} We have $g_1^\#\geq g_2^\#\Rightarrow g_1\geq g_2$. By monotonicity of $R^\#$ we obtain  $ R^\#(g_1,g_1)\geq R^\#(g_2,g_2)$. The claim then comes immediately by representation \eqref{explicit formula for f}.
\end{proof}

\subsection{The Kaniel-Shinbrot iteration}
Now, we will use well-posedness of the associated linear problem and the estimates developed in Section \ref{sec: properties of gain and loss} to prove convergence of the Kaniel-Shinbrot iteration to the unique solution of \eqref{generalized boltzmann equation} in the class of functions bounded by a Maxwellian, if an appropriate beginning condition is satisfied. 

 Let $0<T\leq\infty$ and $\alpha,\beta>0$. Consider a function $f_0\in\mathcal{M}_{\alpha,\beta}^+$ and a pair of functions  $(l_0^\#,u_0^\#)\in\mathcal{M}_{\alpha,\beta}^+\times\mathcal{M}_{\alpha,\beta}^+$. By part \textit{(ii) of Lemma \ref{bound on R tilde lemma}} we have that $G^\#(l_0,l_0,l_0), G^\#(u_0,u_0,u_0)\in L^\infty([0,T),L^{1,+}_{x,v})$. Applying Proposition \ref{linear prop} with $h$ being either $G(l_0,l_0,l_0)$ or $G(u_0,u_0,u_0)$, we find unique functions $l_1, u_1$  such that $l_1$ is the mild solution of
 \begin{equation}\label{diffeq for loss n=1}
\begin{aligned}
\frac{\,d l_1}{\,dt}+ v \cdot \nabla_x l_1  &=G(l_{0},l_{0},l_{0}) - L(l_1,u_{0},u_{0}),\\
l_1(0)&=f_0,
\end{aligned}
\end{equation}
and $u_1$ is the mild solution of
\begin{equation}\label{diffeq for gain n=1}
\begin{aligned}
\frac{\,d u_1}{\,dt}+  v\cdot \nabla_x u_1  &=G(u_{0},u_{0},u_{0}) -  L(u_1,l_{0},l_{0}),\\
u_1(0)&=f_0.
\end{aligned}
\end{equation}
We obtain the following result
\begin{theorem}\label{iteration proposition} Let  $0<T\leq\infty$, $\alpha,\beta>0$ and 
$$K_\beta=C_d\bigg[\|b_2\|_{L^1(\mathbb{S}^{d-1})}(\beta^{-d/2}+\frac{1}{d+\gamma_2-1})+\|b_3\|_{L^1(\mathbb{S}^{2d-1})}(\beta^{-d}+\frac{1}{2d+\gamma_3-1})\bigg],$$
be the constant given in \eqref{constant K text}. Consider $f_0\in\mathcal{M}_{\alpha,\beta}^+$ and $(l_0^\#,u_0^\#)\in\mathcal{M}_{\alpha,\beta}^+\times\mathcal{M}_{\alpha,\beta}^+$ with
\begin{equation}\label{smallness of u_0}
\|u_0^\#\|_{\mathcal{M}_{\alpha,\beta}}< \lambda_{\alpha,\beta}, 
\end{equation}
where
\begin{equation}\label{lambda text}
\lambda_{\alpha,\beta}=\min\left\{\frac{\alpha^{1/2}}{24K_\beta},\frac{\alpha^{1/4}}{2\sqrt{6K_\beta}}\right\}.
\end{equation}
Let $l_1,u_1$ be the mild solutions to \eqref{diffeq for loss n=1} and \eqref{diffeq for gain n=1} respectively, and assume that the following beginning condition holds
\begin{equation}\label{beginning condition}
0\leq  l_0^\#\leq l_1^\#(t)\leq u_1^\#(t)\leq u_0^\#,\quad\forall t\in [0,T).
\end{equation}
Then we conclude the following

\textit{(i)} There are unique sequences $(l_n)_{n}$, $(u_n)_{n}$ such that, for any $n\in\mathbb{N}$, $l_n, u_n$ are the mild solution to \eqref{diffeq for loss n}, \eqref{diffeq for gain n} respectively.
Moreover, for any $n\in\mathbb{N}$, we have
\begin{equation}\label{propagation of inequalities}
  0\leq l_0^\#\leq l_1^\#(t)\leq...\leq l_n^\#(t)\leq u_n^\#(t)\leq...\leq u_1^\#(t)\leq u_0^\#,\quad\forall t\in [0,T).
\end{equation}

\textit{(ii)} For all $t\in[0,T)$, the sequences $\left(l_n^\#(t)\right)_{n}$,  $\left(u_n^\#(t)\right)_{n}$ converge in $\mathcal{M}_{\alpha,\beta}$. Let us define  $$l^\#(t):=\lim_{n\to\infty}l_n^\#(t),\quad u^\#(t):=\lim_{n\to\infty}u_n^\#(t),\quad t\in[0,T).$$ Then, we  conclude that
\begin{align*}
&l^\#,u^\#\in C^0([0,T),L^{1,+}_{x,v})\cap L^\infty([0,T),\mathcal{M}_{\alpha,\beta}^+),\\
&L^\#(l,u,u), L^\#(u,l,l), G^\#(l,l,l), G^\#(u,u,u)\in L^\infty([0,T),L^{1,+}_{x,v}),
\end{align*}
 and the following integral equations are satisfied
\begin{align}
l^\#(t)+\int_0^tL^\#(l,u,u)(\tau)\,d\tau&=f_0+\int_0^t G^\#(l,l,l)(\tau)\,d\tau,\quad\forall t\in[0,T),\label{integrated l}\\
u^\#(t)+\int_0^tL^\#(u,l,l)(\tau)\,d\tau&=f_0+\int_0^t G^\#(u,u,u)(\tau)\,d\tau,\quad\forall t\in[0,T).\label{integrated u}
\end{align}

\textit{(iii)} The limits $l,u$ coincide i.e. $u=l$.

\textit{(iv)} Let us define  $f:=l=u$. Then $f$ is the unique mild solution of the binary-ternary  Boltzmann equation \eqref{generalized boltzmann equation} in $[0,T)$, with initial data $f_0\in\mathcal{M}_{\alpha,\beta}^+$ satisfying 
\begin{equation}\label{condition on uniqueness iteration}
|||f^\#|||_{\infty}\leq\|u_0^\#\|_{\mathcal{M}_{\alpha,\beta}}.
\end{equation}
\end{theorem}
 \begin{remark} The uniqueness claimed above holds in the class of solutions satisfying \eqref{condition on uniqueness iteration}.
 \end{remark}
\begin{proof}
\textit{(i)}: We will construct  sequences $(l_n)_n$, $(u_n)_n$ satisfying \eqref{diffeq for loss n}-\eqref{propagation of inequalities} inductively.
\begin{itemize}
\item $n=1$: $l_1$, $u_1$ satisfy \eqref{diffeq for loss n} for $k=1$ by assumption. Moreover \eqref{propagation of inequalities} reduces for $k=1$ to the assumption \eqref{beginning condition}.
\item Assume  we have constructed $l_1,...,l_{n-1},u_1,...,u_{n-1}$ satisfying \eqref{diffeq for loss n} and
\begin{equation}\label{propagation induction k=n-1}
 l_0^\#\leq l_1^\#(t)\leq...\leq l_{n-1}^\#(t)\leq u_{n-1}^\#(t)\leq...\leq u_1^\#(t)\leq u_0^\#,\quad\forall t\in [0,T).
\end{equation}
Let $l_n,u_n$ be the mild solutions  of \eqref{diffeq for loss n}, \eqref{diffeq for gain n} for $k=n$ respectively , given by Proposition \ref{linear prop}.  Having in mind \eqref{propagation induction k=n-1}, in order to  prove \eqref{propagation of inequalities}, it suffices to show
\begin{equation}\label{sufficient condition k=n}
l_{n-1}^\#(t)\leq l_n^\#(t)\leq u_n^\#(t)\leq u_{n-1}^\#(t),\quad\forall t\in [0,T).
\end{equation}
Fix any $t\in[0,T)$. Then \eqref{propagation induction k=n-1} and Proposition \ref{monotonicity proposition sharp}, which gives monotonicity of $G^\#$, yield that for any $t\in[0,T)$, we have
\begin{equation}\label{gain inequalities}
\begin{aligned}
G^\#(l_{n-2},l_{n-2},l_{n-2})(t)&\leq G^\#(l_{n-1},l_{n-1},l_{n-1})(t)\\
&\leq G^\#(u_{n-1},u_{n-1},u_{n-1})(t)\\
&\leq G^\#(u_{n-2},u_{n-2},u_{n-2})(t).
\end{aligned}
\end{equation}
Using \eqref{propagation induction k=n-1}, \eqref{gain inequalities} and Corollary \ref{comparision corollary from prop} with
 $$g_1^\#=u^\#_{n-2},\quad g_2^\#=u^\#_{n-1}, \quad h_1^\#=G^\#(l_{n-2},  l_{n-2},l_{n-2}), h_2^\#=G^\#(l_{n-1}, l_{n-1},l_{n-1}),$$ 
we obtain $l^\#_{n-1}\leq l^\#_{n}.$
 Similarly, using Corollary \ref{comparision corollary from prop} for $g_1^\#=u^\#_{n-1}$, $g_2^\#=l^\#_{n-1}$, $h_1^\#=G^\#(l_{n-1},l_{n-1},l_{n-1})$, $h_2^\#=G^\#(u_{n-1},u_{n-1},u_{n-1})$, we obtain $l^\#_n\leq u_n^\#$, and using it for $g_1^\#=l^\#_{n-1}$, $g_2^\#=l^\#_{n-2}$, $h_1^\#=G^\#(u_{n-1},u_{n-1},u_{n-1})$, $h_2^\#=G^\#(u_{n-2},u_{n-2},u_{n-2})$, we obtain $u_n^\#\leq u_{n-1}^\#$. Condition \eqref{sufficient condition k=n} is proved and the claim follows.
\end{itemize}

\item \textit{(ii)}: To prove convergence, notice that \eqref{propagation of inequalities} implies that, for any $t\in[0,T)$,
 the sequence $(l_n^\#(t))_{n}$  is increasing  and upper bounded and the sequence $(u_n^\#(t))_{n}$ is decreasing and lower bounded, thus they are convergent. Let us define 
\begin{align*} 
 l^\#(t):=\lim_{n\to\infty} l_n^\#(t),\quad u^\#(t):=\lim_{n\to\infty} u_n^\#(t),\quad t\in[0,T).
 \end{align*}
 Since $u_0^\# \in \mathcal{M}^+_{\alpha, \beta}$,  estimate \eqref{propagation of inequalities} actually implies  that the convergence takes place in $\mathcal{M}_{\alpha,\beta}$ and that $l^\#,u^\#\in L^\infty([0,T),\mathcal{M}_{\alpha,\beta}^+)$. Thus relations \eqref{L in L infty} from Lemma \ref{bound on R tilde lemma} imply that 
\begin{equation}\label{operators in Linf}
L^\#(l,u,u), \text{ }L^\#(u,l,l), \text{ } G^\#(l,l,l),  \text{ }G^\#(u,u,u)\in L^\infty([0,T),L^{1,+}_{x,v}).
\end{equation}
  Moreover, since for any $t\in[0,T)$ we have
   $$(l^\#_n,u_{n-1}^\#,u_{n-1}^\#)(t)\overset{\mathcal{M}_{\alpha,\beta}}\longrightarrow (l^\#,u^\#,u^\#)(t),\quad (u^\#_n,l_{n-1}^\#,l_{n-1}^\#)(t)\overset{\mathcal{M}_{\alpha,\beta}}\longrightarrow (u^\#,l^\#,l^\#)(t),$$
    as $n\to\infty$,  Corollary \ref{L1 continuity} implies that for any $t\in[0,T)$, we have
\begin{equation}\label{loss convergence}
L^\#(l_{n},u_{n-1},u_{n-1})(t)\overset{L^1_{x,v}}\longrightarrow L^\#(l,u,u),\quad L^\#(u_{n},l_{n-1},l_{n-1})(t)\overset{L^1_{x,v}}\longrightarrow L^\#(u,l,l).
\end{equation}
 Similarly, for any $t\in[0,T)$, we obtain
 \begin{equation}\label{gain convergence}
G^\#(l_{n-1},l_{n-1},l_{n-1})(t)\overset{L^1_{x,v}}\longrightarrow G^\#(l,l,l),\quad G^\#(u_{n-1},u_{n-1},u_{n-1})(t)\overset{L^1_{x,v}}\longrightarrow G^\#(u,u,u).
\end{equation}
Moreover, by relation \eqref{propagation of inequalities}, monotonicity of $L^\#,G^\#$,  and the fact that $u_0^\#\in\mathcal{M}_{\alpha,\beta}^+$, Lemma \ref{bound on R tilde lemma} implies
\begin{equation}\label{L1 bound n}
\begin{aligned}
L^\#(l_n,u_{n-1},u_{n-1}),G^\#(l_{n-1},l_{n-1},l_{n-1})\in L^\infty([0,T),L^1_{x,v}),\quad\forall n\in\mathbb{N},\\
L^\#(u_n,l_{n-1},l_{n-1}),G^\#(u_{n-1},u_{n-1},u_{n-1})\in L^\infty([0,T),L^1_{x,v}),\quad\forall n\in\mathbb{N}.
\end{aligned}
\end{equation}
Recalling Definition \ref{boltzmann mild solution}, the initial value problems  \eqref{diffeq for loss n}, \eqref{diffeq for gain n} and  the Fundamental Theorem of Calculus imply that
 for all $n\in\mathbb{N}$ we have
\begin{align}
l_n^\#(t)+\int_0^tL^\#(l_n,u_{n-1},u_{n-1})(\tau)\,d\tau&=f_0+\int_0^t G^\#(l_{n-1},l_{n-1},l_{n-1})(\tau)\,d\tau,\quad\forall t\in[0,T),\label{integrated n l}\\
u_n^\#(t)+\int_0^tL^\#(u_n,l_{n-1},l_{n-1})(\tau)\,d\tau&=f_0+\int_0^t G^\#(u_{n-1},u_{n-1},u_{n-1})(\tau)\,d\tau,\quad\forall t\in[0,T).\label{integrated  n u}
\end{align}
Letting $n\to\infty$ and using   the Dominated Convergence Theorem, we obtain \eqref{integrated l}-\eqref{integrated u}. The fact that  $l^\#,u^\#\in C^0([0,T),L^1_{x,v})$ easily follows from \eqref{integrated l}-\eqref{integrated u}.

 \textit{(iii)}: Since $l_n^\#\leq u_n^\#$ by \eqref{propagation of inequalities}, letting $n\to\infty$, we obtain
 \begin{equation}\label{l leq u}
 0\leq l_0^\#\leq l^\#(t)\leq u^\#(t)\leq u_0^\#,\quad\forall t\in[0,T).
 \end{equation}
 Subtracting \eqref{integrated l} from \eqref{integrated u} and using \eqref{l leq u} and the triangle inequality,  we obtain
 \begin{align} 
  |&u^\#(t)-l^\#(t)|\leq\int_0^t|G^\#(u,u,u)(\tau)-G^\#(l,l,l)(\tau)|+|L^\#(l,u,u)(\tau)-L^\#(u,l,l)(\tau)|\,d\tau.\label{integral bound on difference}
 \end{align}
 Let us estimate the right hand side of \eqref{integral bound on difference}.
  Recalling \eqref{G} 
 triangle inequality yields
 \begin{align}
 \int_0^t&|G^\#(u,u,u)(\tau)-G^\#(l,l,l)(\tau)|\,d\tau\leq\int_0^t|G_2^\#(u,u)(\tau)-G_2^\#(l,l)(\tau)|+|G_3^\#(u,u,u)(\tau)-G_3^\#(l,l,l)(\tau)|\,d\tau.\label{first triangle on gain}
 \end{align}
 Bilinearity of $G_2^\#$,  triangle inequality, bound \eqref{bound on binary loss} from Proposition \ref{bounds on operators proposition}, and the right hand side inequality of \eqref{l leq u} yield
 \begin{align}
 \int_0^t&|G_2^\#(u,u)(\tau)-G_2^\#(l,l)(\tau)|\,d\tau\leq \int_0^t|G_2^\#(u-l,u)(\tau)|+|G_2^\#(l,u-l)(\tau)|\,d\tau\nonumber\\
 &\leq K_\beta\alpha^{-1/2}M_{\alpha,\beta}\|u^\#-l^\#\|_{L^\infty([0,T)],\mathcal{M}_{\alpha,\beta})}\left(|||u^\#|||_{\infty}+|||l^\#|||_{\infty}\right)\nonumber\\
 &\leq 2K_\beta\alpha^{-1/2}M_{\alpha,\beta}\|u_0^\#\|_{\mathcal{M}_{\alpha,\beta}}|||u^\#-l^\#|||_{\infty}.\label{triangle on binary gain}
 \end{align}

 Trilinearity of $G_3^\#$, triangle inequality, bound  \eqref{bound on ternary loss} from Proposition \ref{bounds on operators proposition}, and the right hand side of \eqref{l leq u}  yield
\begin{align} 
 \int_0^t&|G_3^\#(u,u,u)(\tau)-G_3^\#(l,l,l)(\tau)|\,d\tau\leq \int_0^t|G_3^\#(u-l,u,u)(\tau)|+|G_3^\#(l,u-l,u)(\tau)|+|G_3^\#(l,l,u-l)(\tau)|\,d\tau\nonumber\\
&\leq K_\beta\alpha^{-1/2}M_{\alpha,\beta}|||u^\#-l^\#|||_{\infty}\left(|||u^\#|||^2_{\infty}+|||u^\#|||_{\infty}|||l^\#|||_{\infty}+|||l^\#|||^2_{\infty}\right)\nonumber\\
&\leq 3K_\beta\alpha^{-1/2}M_{\alpha,\beta}\|u_0^\#\|_{\mathcal{M}_{\alpha,\beta}}^2|||u^\#-l^\#|||_{\infty}\label{triangle on ternary gain}.
 \end{align}
 Then estimates \eqref{triangle on binary gain}-\eqref{triangle on ternary gain}   yield
 \begin{equation}\label{estimate on gain integral pre}
  \int_0^t|G^\#(u,u,u)(\tau)-G^\#(l,l,l)(\tau)|\,d\tau\leq 6K_\beta\alpha^{-1/2}M_{\alpha,\beta}(\|u_0^\#\|_{\mathcal{M}_{\alpha,\beta}}+\|u_0^\#\|_{\mathcal{M}_{\alpha,\beta}}^2)|||u^\#-l^\#|||_{\infty}.
 \end{equation}
  By a similar argument, using \eqref{bound on binary loss}, \eqref{bound on ternary loss} instead, we  also have
 \begin{equation}\label{estimate on loss integral pre}
\int_0^t |L^\#(l,u,u)(\tau)-L^\#(u,l,l)(\tau)|\,d\tau\leq 6K_\beta\alpha^{-1/2}M_{\alpha,\beta}(\|u_0^\#\|_{\mathcal{M}_{\alpha,\beta}}+\|u_0^\#\|_{\mathcal{M}_{\alpha,\beta}}^2)|||u^\#-l^\#|||_{\infty}.
 \end{equation}
 Combining \eqref{integral bound on difference}, \eqref{estimate on gain integral pre}-\eqref{estimate on loss integral pre}, we obtain
 \begin{equation*}
 |u^\#(t)-l^\#(t)|\leq  12K_\beta\alpha^{-1/2}M_{\alpha,\beta}(\|u_0^\#\|_{\mathcal{M}_{\alpha,\beta}}+\|u_0^\#\|_{\mathcal{M}_{\alpha,\beta}}^2)|||u^\#-l^\#|||_{\infty},\quad\forall t\in[0,T),
 \end{equation*}
 which is equivalent to 
\begin{equation}\label{equivalent norm}
 |||u^\#-l^\#|||_{\infty}\leq 12K_\beta\alpha^{-1/2}(\|u_0^\#\|_{\mathcal{M}_{\alpha,\beta}}+\|u_0^\#\|_{\mathcal{M}_{\alpha,\beta}}^2)|||u^\#-l^\#|||_{\infty}.
\end{equation}
Notice though that \eqref{smallness of u_0}-\eqref{lambda text} yield
\begin{equation*}
12K_\beta\alpha^{-1/2}(\|u_0^\#\|_{\mathcal{M}_{\alpha,\beta}}+\|u_0^\#\|_{\mathcal{M}_{\alpha,\beta}}^2)<1,
\end{equation*}
hence \eqref{equivalent norm} yields $u=l$.

 \textit{(iv)}: To prove existence, let us define $f$ by $f^\#:=l^\#=u^\#\in C^0([0,T),L^{1,+}_{x,v})\cap L^\infty([0,T),\mathcal{M}_{\alpha,\beta}^+)$.   Then,  either \eqref{integrated l} or \eqref{integrated u} implies
\begin{equation*}
f^\#(t)+\int_0^tL^\#(f,f,f)(\tau)\,d\tau=f_0+\int_0^tG^\#(f,f,f)(\tau)\,d\tau,\quad\forall t\in[0,T),
\end{equation*}
therefore
\begin{equation*}
\begin{cases}
\displaystyle\frac{\,d f^\#}{\,dt}+L^\#(f,f,f)=G^\#(f,f,f),\\
f^\#(0)=f_0.
\end{cases}
\end{equation*}
Recalling Definition \ref{boltzmann mild solution}, we conclude that $f$ is a mild solution to the binary-ternary  Boltzmann equation \eqref{generalized boltzmann equation} with initial data $f_0$. Bound \eqref{condition on uniqueness iteration} directly follows from \eqref{l leq u}.

Uniqueness of solutions satisfying \eqref{condition on uniqueness iteration} follows similarly to the proof of $(iii)$ using a bilinearity-trilinearity argument and Proposition \ref{bounds on operators proposition}. Clearly, condition \eqref{condition on uniqueness iteration} is needed to have a contraction.
\end{proof}
\section{Global well-posedness near vacuum}\label{sec: gwp}
In this final section, we prove the main result of this paper, stated in Theorem \ref{gwp theorem}, which gives global well-posedness of \eqref{generalized boltzmann equation} near vacuum in the interval $[0,T)$, where $0<T\leq\infty$. To prove this result we will rely on the time average bound of the gain term from Proposition \ref{bounds on operators proposition}.

\textbf{ Proof of Theorem \ref{gwp theorem}}
 
 Consider $f_0\in\mathcal{M}_{\alpha,\beta}^+$ satisfying \eqref{condition for existence} and let us define $l_0^\#=0$, $u_0^\#=C_{out}M_{\alpha,\beta}$, where
 \begin{equation}\label{C_out}
 C_{out}=\frac{1-\sqrt{1-48K_\beta\alpha^{-1/2}(1+\frac{\alpha^{1/4}}{2\sqrt{6K_\beta}})\|f_0\|_{\mathcal{M}_{\alpha,\beta}}}}{24K_\beta\alpha^{-1/2}\left(1+\frac{\alpha^{1/4}}{2\sqrt{6K_\beta}}\right)},
 \end{equation}
 $K_\beta$ is given by \eqref{constant wide K intro}.
 The reasoning behind defining $C_{out}$ will become clear in \eqref{equation of C_out}.
 Notice that due to \eqref{condition for existence} $u_0^\#$ is well defined. In order to conclude the proof, we will use Theorem \ref{iteration proposition}. Recalling 
 $$\lambda_{\alpha,\beta}=\min\left\{\frac{\alpha^{1/2}}{24K_\beta},\frac{\alpha^{1/4}}{2\sqrt{6K_\beta}}\right\},$$
  from \eqref{lambda text},  \eqref{C_out} and \eqref{condition for existence}  yield
 \begin{equation}\label{u_0<lam}
 \|u_0^\#\|_{\mathcal{M}_{\alpha,\beta}}=C_{out}<\lambda_{\alpha,\beta},
 \end{equation} 
 thus the conditions of Theorem  \ref{iteration proposition} are satisfied.
 By Theorem \ref{iteration proposition}, it suffices to prove that  the beginning condition \eqref{beginning condition} for the approximating sequences  generated  by $f_0\in\mathcal{M}_{\alpha,\beta}^+$ and the pair of functions $(l_0^\#,u_0^\#)\in\mathcal{M}_{\alpha,\beta}^+\times\mathcal{M}_{\alpha,\beta}^+$ is satisfied.
Indeed, by the iteration scheme \eqref{diffeq for loss n=1}, we have
\begin{equation*}
\begin{aligned}
&\frac{\,dl_1^\#}{\,dt}+l_1^\#R^\#(u_0,u_0)=0,\\
&\frac{\,du_1^\#}{\,dt}=G^\#(u_0,u_0,u_0),\\
&u_1^\#(0)=l_1^\#(0)=f_0,
\end{aligned}
\end{equation*}
 therefore, we obtain
 \begin{align}
 l_1^\#(t)&=f_0\exp\left(-\int_0^t R^\#(u_0,u_0)(\tau)\,d\tau\right),\quad t\in[0,T),\label{formula for l1}\\
 u_1^\#(t)&=f_0+\int_0^t G^\#(u_0,u_0,u_0)(\tau)\,d\tau,\quad t\in[0,T).\label{formula fo u1}
 \end{align}
 Since $u_0\geq 0$, formulas \eqref{formula for l1}-\eqref{formula fo u1} together with Proposition \ref{monotonicity proposition sharp} imply
\begin{equation}\label{first part of the claim} 
 0=l_0^\#\leq l_1^\#(t)\leq u_1^\#(t),\quad\forall t\in[0,T).
 \end{equation}
   It remains to prove that 
 \begin{equation}\label{u1 leq u0}
 u_1^\#(t)\leq u_0^\#,\quad\forall t\in[0,T).
 \end{equation}
   By representation \eqref{formula fo u1} and \eqref{bound on total loss} from Proposition \ref{bounds on operators proposition}, we obtain
\begin{align}
u_1^\#(t)&\leq \|f_0\|_{\mathcal{M}_{\alpha,\beta}}M_{\alpha,\beta}+K_\beta\alpha^{-1/2} M_{\alpha,\beta}\|u_0^\#\|^2_{\mathcal{M}_{\alpha,\beta}}(1+\|u_0^\#\|_{\mathcal{M}_{\alpha,\beta}})\nonumber\\
&\leq M_{\alpha,\beta}\left[\|f_0\|_{\mathcal{M}_{\alpha,\beta}}+K_\beta\alpha^{-1/2}\left(1+\frac{\alpha^{1/4}}{2\sqrt{6K_\beta}}\right) C_{out}^2\right],\label{use of assumption}
\end{align}
where to obtain \eqref{use of assumption}, we use the fact that $u_0^\#=C_{out}M_{\alpha,\beta}$ and \eqref{u_0<lam}. Recalling \eqref{C_out}, we notice that $C_{out}$ satisfies the equation
\begin{equation}\label{equation of C_out}
 \|f_0\|_{\mathcal{M}_{\alpha,\beta}}+12K_\beta\alpha^{-1/2}\left(1+\frac{\alpha^{1/4}}{2\sqrt{6K_\beta}}\right)C_{out}^2=C_{out},
 \end{equation}
 thus \eqref{use of assumption} implies
 \begin{equation*}
 u_1^\#(t)\leq C_{out}M_{\alpha,\beta}=u_0^\#,\quad\forall t\in[0,T).
 \end{equation*}
Estimate \eqref{u1 leq u0} is proved and the claim of Theorem \ref{gwp theorem} follows.


\begin{thebibliography}{99}
 \bibitem{alonso} R. Alonso, {\em Existence of global solutions to the Cauchy problem for the
inelastic Boltzmann equation with near-vacuum data}, Indiana University Mathematics Journal
Vol. 58, No. 3 (2009), pp. 999-1022.
 \bibitem{alonso-gamba} R. Alonso, I. M. Gamba, {\em Distributional and Classical Solutions
to the Cauchy Boltzmann Problem for Soft Potentials
with Integrable Angular Cross Section}, J Stat Phys (2009) 137: 1147–1165
DOI 10.1007/s10955-009-9873-3.
\bibitem{thesis} I. Ampatzoglou, {\em Higher order extensions of the Boltzmann equation, Ph.D. dissertation}, Dept. of Mathematics, UT Austin (2020).
\bibitem{future AP} I. Ampatzoglou, N. Pavlovi\'c, {\em Rigorous derivation of a binary-ternary Boltzmann equation for a dense gas of hard spheres}, 	arXiv:2007.00446 (2020).
\bibitem{ours} I. Ampatzoglou, N. Pavlovi\'c, {\em Rigorous derivation of a ternary Boltzmann equation for a classical system of particles}, to appear in Comm. Math. Phys. arXiv:1903.04279 (2021).
\bibitem{beto85} N. Bellomo, G. Toscani, {\em On the Cauchy problem for the nonlinear Boltzmann equation: global existence, uniqueness and asymptotic stability}, J. Math. Phys. 26 (1985), no. 2, 334–338. 
\bibitem{multiple gamba 2} A.V. Bobylev, I.M. Gamba, C. Cercignani, {\em On the self-similar asymptotics for generalized non-linear kinetic Maxwell models}, Commun. Mathematical Physics 291, 599 - 644 (2009).
\bibitem{multiple gamba} A.V. Bobylev, I.M. Gamba, C. Cercignani, {\em Generalized kinetic Maxwell type models of granular gases}, Mathematical models of granular matter Series: Lecture Notes in Mathematics Vol.1937, Springer, G. Capriz, P. Giovine and P. M. Mariano (Eds.) (2008) ISBN: 978-3-540-78276-6.
\bibitem{chyu08} Y.-K. Cho, B.-J. Yu, {\em Uniform stability estimates for solutions and their gradients to the Boltzmann equation: a unified approach}, J. Differential Equations 245 (2008), no. 12, 3615–3627. 
\bibitem{gl06} R. T. Glassey, {\em Global solutions to the Cauchy problem for the relativistic Boltzmann equation with near-vacuum data}, Comm. Math. Phys. 264 (2006), no. 3, 705–724.
\bibitem{hano10} S.-Y. Ha, S.-E. Noh, {\em Global existence and stability of mild solutions to the inelastic Boltzmann system for gas mixtures}, Quart. Appl. Math. 68 (2010), no. 4, 671–699. 
\bibitem{hanoyu07} S.-Y. Ha, S.-E. Noh, S.-B. Yun, {\em Global existence and stability of mild solutions to the Boltzmann system for gas mixtures}, Quart. Appl. Math. 65 (2007), no. 4, 757–779. 
\bibitem{illner-shinbrot} R. Illner, M. Shinbrot, {\em The Boltzmann Equation: Global Existence for a Rare Gas in an Infinite Vacuum}, Commun. Math. Phys. 95, 217-226 (1984).
\bibitem{kaniel-shinbrot} S. Kaniel, M. Shinbrot, {\em The Boltzmann Equation: Uniqueness and Local Existence}, Commun. math. Phys. 58, 65-84 (1978).
\bibitem{mipe97} S. Mischler, B. Perthame, {\em Boltzmann equation with infinite energy: renormalized solutions and distributional solutions for small initial data and initial data close to a Maxwellian}, SIAM J. Math. Anal. 28 (1997), no. 5, 1015–1027. 
\bibitem{pato89} A. Palczewski, G. Toscani, {\em Global solution of the Boltzmann equation for rigid spheres and initial data close to a local Maxwellian}, J. Math. Phys. 30 (1989), no. 10, 2445–2450. 
\bibitem{po88} J. Polewczak {\em Classical solution of the nonlinear Boltzmann equation in all $\mathbb{R}^3$: asymptotic behavior of solutions}, J. Statist. Phys. 50 (1988), no. 3-4, 611–632
\bibitem{st10} R. M. Strain, {\em Global Newtonian limit for the relativistic Boltzmann equation near vacuum}, SIAM J. Math. Anal. 42 (2010), no. 4, 1568–1601. 
\bibitem{to86} G. Toscani, {\em On the nonlinear Boltzmann equation in unbounded domains}, Arch. Rational Mech. Anal. 95 (1986), no. 1, 37–49. 
\bibitem{to88} G. Toscani, {\em Global solution of the initial value problem for the Boltzmann equation near a local Maxwellian}, Arch. Rational Mech. Anal. 102 (1988), no. 3, 231–241. 
\bibitem{beto87} G. Toscani, N. Bellomo {\em The Enskog-Boltzmann equation in the whole space $\mathbb{R}^3$: some global existence, uniqueness and stability results}, Comput. Math. Appl. 13 (1987), no. 9-11, 851–859.
\bibitem{wezh06} J. Wei, X. Zhang, {\em Eternal solutions of the Boltzmann equation near travelling Maxwellians}, J. Math. Anal. Appl. 314 (2006), no. 1, 219–232. 
\bibitem{wezh14} J. Wei, X. Zhang, {\em Distributional solutions of the Boltzmann equation with infinite energy}, J. Math. Anal. Appl. 417 (2014), no. 1, 327–335.
\bibitem{zhhu08} X, Zhang, S. Hu, {\em Solutions of the Boltzmann equation with infinite energy: existence, stability and long time behavior}, J. Math. Anal. Appl. 347 (2008), no. 1, 26–34. 
\end{thebibliography}
\end{document}